\documentclass[a4paper,reqno]{amsart}
\usepackage{amscd}
\usepackage{amssymb,latexsym}
\usepackage[centertags]{amsmath}
\usepackage{amsfonts}
\usepackage{amssymb}
\usepackage{amsthm}
\usepackage{newlfont}
\usepackage[all]{xy}
\usepackage{pifont}
\usepackage{geometry}
\usepackage{graphicx}
\usepackage{txfonts}
\usepackage{verbatim}
\usepackage{fullpage}

\usepackage{hyperref}
\usepackage{cite}

\usepackage{xcolor}

\def\ts{ {}^{t}\hspace*{-.05cm}  \star   }

\makeatletter
\def\thmhead@plain#1#2#3{%
  \thmname{#1}\thmnumber{\@ifnotempty{#1}{ }\@upn{#2}}%
  \thmnote{ {\the\thm@notefont#3}}}
\let\thmhead\thmhead@plain
\makeatother

 \newtheorem{teo}{Theorem}[section]
 \newtheorem{cor}[teo]{Corollary}
 \newtheorem{lem}[teo]{Lemma}
 \newtheorem{pro}[teo]{Proposition}

 \theoremstyle{definition}
 \newtheorem{defi}[teo]{Definition}
 \newtheorem{exa}[teo]{Example}
 \newtheorem{rmk}[teo]{Remark}
 
\def\3{{\sf 3}}

\def\A{\mathcal{A}}

\DeclareMathOperator{\op}{op}

\newcommand{\calR}{\mathcal{R}}

\newcommand{\calA}{\mathcal{A}}

\DeclareMathOperator{\End}{\mathrm{End}}

\DeclareMathOperator{\rmL}{\mathrm{L}}
\DeclareMathOperator{\rmR}{\mathrm{R}}
\DeclareMathOperator{\rmD}{\mathrm{D}}

\DeclareMathOperator{\ann}{Ann}
\DeclareMathOperator{\ad}{ad}
\DeclareMathOperator{\sgn}{sgn}

\DeclareMathOperator{\asso}{Asso}

\begin{document}

\title[Bialgebra theory for nearly associative algebras and $LR$-algebras]
{Bialgebra theory for nearly associative algebras and $LR$-algebras: equivalence, characterization, and $LR$-Yang-Baxter Equation}

\author[E.~Barreiro]{Elisabete~Barreiro}

\address{Elisabete~Barreiro, University of Coimbra, Department of Mathematics, CMUC,
Largo D. Dinis
3000-143 Coimbra, Portugal.
{\em E-mail address}: {mefb@mat.uc.pt}}



\author[S.~Benayadi]{Sa\"{i}d~Benayadi}

\address{Sa\"{i}d~Benayadi, Universit\'{e} de Lorraine, Laboratoire IECL, CNRS-UMR 7502, UFR MIM, 3 rue Augustin Frenel, BP 45112, 57073 Metz Cedex 03, France. {\em E-mail address}: {said.benayadi@univ-lorraine.fr}}


\author[C.~Rizzo]{Carla Rizzo}

\address{Carla~Rizzo, University of Coimbra, Department of Mathematics, CMUC,
Largo D. Dinis
3000-143 Coimbra, Portugal. {\em E-mail address}: {carlarizzo@mat.uc.pt}}

\thanks{E.~Barreiro  and C.~Rizzo were partially supported by the Centre for Mathematics of the University of Coimbra (funded by the Portuguese Government through FCT/MCTES, DOI 10.54499/UIDB/00324/2020)}



\begin{abstract}
We develop the bialgebra theory for two classes of non-associative algebras: nearly associative algebras and $LR$-algebras. In particular, building on recent studies that reveal connections between these algebraic structures, we establish that nearly associative bialgebras and $LR$-bialgebras are, in fact, equivalent concepts. We also provide a cha\-racterization of these bialgebra classes based on the coproduct.
Moreover, since the development of nearly associative bialgebras — and by extension, $LR$-bialgebras — requires the framework of nearly associative $L$-algebras, we introduce this class of non-associative algebras and explore their fundamental properties.
Furthermore, we identify and characterize a special class of nearly associative bialgebras, the coboundary nearly associative bialgebras, which provides a natural framework for studying the Yang-Baxter equation (YBE) within this context. 
\end{abstract}

\maketitle

\section{Introduction}
  
A nearly associative algebra is a non-associative algebra, i.e., a not necessarily associative algebra, over a field $\mathbb{K}$ satisfying the multilinear identity in three variables: $(xy)z=y(zx)$. This class of non-associative algebras was first approached in \cite{DS} and, then, has been deeply studied in \cite{BBR2024}.   Further information on this class can be found in \cite{AKS}.  
Another significant class of algebras is the one of $LR$-algebras, also known as bicommutative algebras. These algebras are non-associative algebras over a field $\mathbb{K}$ satisfying both right-commutativity and left-commutativity, i.e., the identities $(xy)z=(xz)y$ and $x(yz)=y(xz)$. In particular, a non-associative algebra is called an $R$-algebra if it satisfies only right-commutativity, and an $L$-algebra if it satisfies only left-commutativity.
One-sided commutative algebras appeared first in the paper \cite{Cayley1857} due to Cayley in 1857, and they emerge in contexts such as the study of hydrodynamic type equations \cite{BalNovi1985}. In recent years, the theory of $LR$-algebras has been intensively studied  (see for example  \cite{Burde, Drensky2019, Drensky2018, DIT2011, DT2003}).
Even though in general the above two classes of non-associative algebras are distinct, the authors in \cite{BBR2024} show a connection between them. They prove that if $A$ is a non-associative algebra equipped with an invariant, non-degenerate, and symmetric bilinear form (referred to as quadratic algebra), then $A$ is nearly associative if and only if it is an $LR$-algebra.

Bialgebras are algebraic structures that arise from a combination of a multiplication and a comultiplication, subject to specific compatibility conditions. Notable examples of bialgebras include associative bialgebras and Lie bialgebras.
Drinfeld's introduction of Lie bialgebras \cite{Dri1983} has had a significant impact, particularly in the context of integrable models, where they are intimately connected to the classical Yang-Baxter equation. For associative algebras, Joni and Rota \cite{JR1995} introduced the concept of an infinitesimal bialgebra to provide a formal framework for the calculus of divided differences. Subsequently, Aguiar expanded this idea into a comprehensive theory of infinitesimal bialgebras, drawing parallels with the well-established theory of Lie bialgebras \cite{Agu2000, Agu2001}.
Zhelyabin generalized Drinfeld's construction to an arbitrary variety of non-associative algebras and studied the Jordan case (cf. \cite{Zhe1997}).  We also refer the interested reader to \cite{Bai2008, BB2009, BB2016, DH2016} for more bialgebra structure results on other classes of non-associative algebras.

The goal of this paper is to develop a bialgebra theory specifically for nearly associative algebras. 
We use Drinfeld's construction of the double of a Lie algebra, generalized by Zhelyabin to any non-associative algebra, to define the class of nearly associative bialgebras. But, since double of algebras are quadratic algebras, the notions of nearly associative bialgebra and $LR$-bialgebra are equivalent, effectively extending the theory to the class of $LR$-algebras as well. We further broaden the theory by proving the equivalence of nearly associative bialgebras with $R$-bialgebras and $L$-bialgebras.
Moreover, since such a theory is only feasible in the context of nearly associative $L$-algebras, i.e., nearly associative algebras satisfying left-commutativity, we introduce this class of algebras and investigate some of its fundamental properties.
Furthermore, given the importance of the classical Yang-Baxter equation and coboundary Lie bialgebras, we define and explore nearly associative analogs of these concepts.
It should be noted that in \cite{DS} there is an attempt at a definition of nearly associative bialgebras using the notion of matched pairs of nearly associative algebras.

The paper is structured as follows. Section~\ref{Section: Basic definitions and results} provides the necessary background on nearly associative algebras and $LR$-algebras, covering basic definitions and properties. In Section~\ref{Section: Representation of NA algebras}, we present the representation theory of nearly associative algebras. We specifically define both the adjoint and coadjoint representations, establish necessary and sufficient conditions for the existence of the coadjoint representation, and present a general result applicable to any non-associative algebra (Corollary \ref{cor:existence of quadratic}). 
Section~\ref{Section: NAL-algebras} is devoted to the class of nearly associative $L$-algebras and presents some of their structural properties.
In Section~\ref{Section: Coalgebras}, we discuss  nearly coassociative coalgebras, $L$-coalgebras, and $R$-coalgebras, along with their properties.
In Section~\ref{Section: NA bialgebras}, we introduce nearly associative bialgebras through the double construction of nearly associative $L$-algebras in the sense of Drinfeld. We establish the equivalence between nearly associative bialgebras and the concepts of $R$-bialgebra, $L$-bialgebra, and $LR$-bialgebra, and provide a characterization of nearly associative bialgebras based on the coproduct (Theorem \ref{teo: nearly ass bialgebras comultiplication}). Finally, Section~\ref{Section: Coboundary NA bialgebras} is dedicated to coboundary nearly associative bialgebras, which serve as a natural framework for studying analogs of the classical Yang-Baxter equation and $r$-matrices in this context.

Throughout this paper, all vector spaces are assumed to be finite-dimensional over field $\mathbb{K}$ of characteristic zero.
We will always work with non-associative algebras, which we sometimes refer to just as algebras for simplicity.

\section{Basic definitions and results}\label{Section: Basic definitions and results}

In this section, we present some fundamental definitions and results required for the following sections.
Let $(A, \cdot )$ be an algebra over $\mathbb{K}$, where the product is often denoted by juxtaposition to simplify the presentation when confusion does not arise.
As usual, we denote by $ \rmL,\rmR : A \longrightarrow \End(A)$ the operators of left and right multiplications, respectively, i.e., 
for $x \in  A$, the \textit{left} (resp. \textit{right}) \textit{multiplication} by $x$ is the endomorphism $\rmL_x =\rmL(x):    A \longrightarrow  A$ (resp. $\rmR_x =\rmR(x):   A \longrightarrow  A$) of $ A$ defined by 
 \begin{equation*}
\begin{split}
\rmL_x(y):=x y \hspace*{1cm}
\end{split}
\begin{split}
\Big(\mbox{resp. }  \rmR_x(y):=y  x \Big),
\end{split}
\end{equation*}
for $y\in A$.
We write the usual composition of endomorphism by juxtaposition.

The \textit{associator}  of $A$ is the trilinear map  $\mbox{Asso}: A \times A \times  A \longrightarrow  A$ defined by
\begin{equation*}
\asso(x,y,z):=(x y) z-x(y z),
 \end{equation*}
for $x,y,z\in  A$.  If for all $x,y,z\in A$, $\asso(x,y,z)=0$, $A$ is {\it associative}, otherwise $A$ is {\it not associative}. Also, recall that a non-associative algebra $A$ is said \textit{flexible} if $\asso(x,y,x)=0$ for all $x,y \in A$, or equivalently, if $\asso(x,y,z)+\asso(z,y,x)=0$  for all $x,y,z \in A$.
The {\it annihilator} of $A$ is $$\ann(A):=\{x \in A \, | \, xy=yx=0 \mbox{ for all } y \in A\}.$$
Since $\mathbb{K}$ has characteristic different from $2$, we may define the following  two new products on the  underline vector space $A $: for $  x,y \in A $,
 \begin{eqnarray}
 && [x,y]:= \frac{1}{2}( x y-y x ),\,\,\,\, x\bullet y:=  \frac{1}{2}( x y+y x). 
 \label{eq:A-A+}
 \end{eqnarray}
We denote ${A}^-:= ({A}, [-,-])$ and ${A}^+:=({A},  \bullet  )$. 
The product of the algebra $A$ can be written in terms of  these two new products as follows:  for $  x,y \in A $,
 \begin{eqnarray}
 && xy  = [x,y] + x\bullet y.
 \label{eq:cdot}
 \end{eqnarray}
We denote the associator of ${A}^+ =({A},  \bullet  )$ by $\asso^+$.

\begin{defi}
A non-associative algebra $\calA$ is said \textit{nearly associative} if
\begin{equation}
 \label{eq: nearly associative algebra}
        (x y)  z=y (z x),
    \end{equation}
for all $x,y,z\in \calA$.
\end{defi}

\begin{rmk}    
\label{rmk: left and righ multiplication}
    If $\calA$ is a nearly associative algebra, then, for all $ x,y \in\calA$,
    \begin{eqnarray}
  && \rmL_x   \rmL_y=\rmR_y   \rmR_x,  \label{eq:lrm1}\\
&& \rmL_x   \rmR_y=\rmL_{y  x},  \label{eq:lr2}\\
 && \rmR_x   \rmL_y=\rmR_{x  y}. \label{eq:lr3}  
 \end{eqnarray}
 Conversely, if a non-associative algebra $A$ satisfies any of the Conditions \eqref{eq:lrm1}, \eqref{eq:lr2}, or \eqref{eq:lr3}, then $A$ is a nearly associative algebra.
\end{rmk}

\begin{pro}[{\cite[Proposition 4.1]{BBR2024}}]\label{pro: consequnces def NA}
    Let $\calA$ be a nearly associative algebra. Then  we have
\begin{eqnarray}
&& \asso(x,y,z)=2[y,z x],  \label{eq:condit1}
\end{eqnarray}
for all $x,y,z \in \calA$.
\end{pro}
The following gives a characterization of flexible nearly associative algebras.
\begin{pro}[{\cite[Proposition 5.4]{BBR2024}}]\label{pro: A flexible and center condition}
    Let $\calA$ be a nearly associative algebra. Then $\calA$ is  flexible if and only if $\calA \bullet \calA \subseteq Z(\calA^-)$.
\end{pro}

Recall that an algebra $  (\mathfrak{g}, [-,-])$ is called a \emph{Lie algebra} if it satisfies the following two conditions: 
\begin{align*}
 [x,y]= -[y,x] , & \\  
[x,[y,z]]+[y,[z,x]]   +[z,[x,y]]=0, &\; \; \; \mbox{ ({\it  Jacobi identity}) }  
\end{align*}
for all $x,y,z\in \mathfrak{g} $. As usual, the \textit{center} of a Lie algebra $  (\mathfrak{g}, [-,-])$ is defined by 
 $$
Z(\mathfrak{g}):=\{x\in \mathfrak{g}\,|\,  [x,y]=0 \mbox{ for all } y\in \mathfrak{g}\}.
$$

An algebra $(A ,\bullet)$ is a \emph{Jordan algebra} if the following two conditions hold: 
\begin{align*}
  x \bullet y =  y \bullet x   ,& \\   
 (x \bullet y) \bullet x^2 =x \bullet (  y \bullet x^2 ),&  \; \; \; \mbox{ ({\it Jordan identity}) }  \label{jor id1}
\end{align*}
\noindent for all $x,y\in {A} $.

\begin{defi} A  non-associative algebra $A$ is called: 
 \begin{enumerate}
 \item \textit{Lie admissible} if  ${A}^- = ({A}, [-,-])$ is a Lie  algebra.
 \item \textit{Jordan admissible} if  ${A}^+ =({A},  \bullet  )$ is a Jordan  algebra.
\item {\it Lie-Poisson-Jordan admissible} if it is Lie and Jordan admissible, and also for any $x,y,z\in A$
$$[x,y \bullet z]= [x,y] \bullet z + y \bullet [x,y].$$
\end{enumerate}
\end{defi}

\noindent For a Lie admissible algebra $A$, the Lie algebra ${A}^- = ({A}, [-,-])$ is called an \textit{underlying Lie algebra}  of $A$. 
For a Jordan admissible algebra $A$, the Jordan algebra ${A}^+ =({A},  \bullet  )$ is called an \textit{underlying Jordan algebra}  of $A$.

Recall that a derived series of an algebra $ A$ is the series $ A \supseteq A^{(1)} \supseteq A^{(2)}\supseteq \cdots$, where $A^{(1)}:=A^2$ and $A^{(i+1)}:=(A^{(i)})^2$, for $i\geq 1$. Then $ A$ is called \textit{solvable} if for some $r\geq 1$ we have $A^{(r)}=\{0\}$. 

\begin{pro}[{\cite[Propositions 4.2 and 4.5]{BBR2024}}] \label{pro: A^+ Jordan A^- Lie}
Let $\calA$ be a nearly associative algebra. Then ${\calA}^- $ is a solvable Lie algebra and $\A^+$ is a Jordan algebra.
\end{pro}

The following proposition characterizes the nearly associative algebras.

\begin{pro}[{\cite[Proposition 4.7]{BBR2024}}] \label{very}
A non-associative algebra $\calA$ is a nearly associative algebra if and only if the following three conditions hold:
\begin{enumerate}
\item  ${\calA}^-$ is a Lie algebra;
\item For $x,y,z\in\calA$,   $ \asso^+(x,y,z) =[y,[z,x]] ;$  
\item For   $x,y,z\in\calA$, $ [x,y]\bullet z+ [ x\bullet y , z]+ [x,z]\bullet y  +[  x\bullet z ,y]= 0. $
\end{enumerate}
\end{pro}


\smallskip

Now we introduce other classes of non-associative algebras that are crucial to study nearly associative bialgebras.

\begin{defi}\label{def LR-algebra}
    A non-associative algebra $A$ is called:
\begin{itemize}
 \item \textit{$R$-algebra} if $\rmR_x   \rmR_y=\rmR_y   \rmR_x$ for all  $x,y \in A$, i.e., 
 for all $x,y,z\in A$ 
   \begin{equation}
   (x y)z = (xz)y. \label{condition: R-alg}
   \end{equation}
\item \textit{$L$-algebra} if  $\rmL_x   \rmL_y=\rmL_y   \rmL_x$ for all  $x,y \in A$, i.e., 
for all $x,y,z\in A$
    \begin{equation}
    x(yz)=y(xz). \label{condition: L-alg}
    \end{equation}
 \item \textit{$LR$-algebra} if $A$ is $L$-algebra and $R$-algebra simultaneously, that is,  
 Conditions \eqref{condition: R-alg} and \eqref{condition: L-alg} are both satisfied.
\end{itemize}
\end{defi}

\begin{rmk}
    Any associative $LR$-algebra is a nearly associative algebra.
\end{rmk}

 Given an algebra $(A,\cdot),$ the \emph{opposite algebra} $A^{\op}=(A, \cdot^{\op})$ is the underlying vector space of $A$ endowed with the \emph{opposite product} $x\cdot^{\op} y:=y\cdot x$ for $x,y\in A.$
 
\begin{rmk}\label{rmk: opposite algebra}
   The opposite algebra of a nearly associative algebra is still a nearly associative algebra. Moreover, the opposite algebra of an $L$-algebra is an $R$-algebra, and vice versa. 
\end{rmk}

Now, we present an example of a six-dimensional nearly associative algebra that is not an $LR$-algebra.

\begin{exa}\label{exa:NA not LR}
    Let $\calA$ be a vector space of dimension six with basis $\{e_1,e_2,e_3,e_4,e_5,e_6 \} $. 
     If we equipped $\calA$ with the linear product defined by
    \begin{align*}
          & \qquad  e_1 e_2=e_3, \ \ e_2 e_1 =-e_4, \ \ e_2 e_2 =e _5, \ \ e_1 e_5 = -e_6,\ \ e_5 e_1 =   e_2 e_3 = e_4 e_2 = e_6, \ \ e_2 e_4 = -2 e_6, \ \ e_3 e_2 = 2 e_6, 
    \end{align*}
     and all other products equal to $0$, then  $\calA$ is a nearly associative algebra (cf. \cite[Example 3.17]{BBR2024}). Since $(e_2 e_1) e_2 =-e_6 \neq e_6= (e_2 e_2) e_1$ and $e_2 (e_1 e_2) = e_6 \neq - e_6 =e_1 ( e_2 e_2)$, $\calA$ is neither an $L$-algebra nor an $R$-algebra, and as a consequence it is not an $LR$-algebra.
\end{exa}

Next, we present an example of a three-dimensional $LR$-algebra that is not a nearly associative algebra.

\begin{exa}\label{exa:LR not NA}  
    Let $\calA$ be a vector space of dimension three with basis $\{e_1,e_2,e_3\} $. 
     If we equipped $\calA$ with the linear product defined by
    \begin{align*}
          & \qquad  e_2 e_1=-e_3,  \ \ e_2 e_2 =e _2, \ \ e_2 e_3 = e_3 e_2= e_3,
    \end{align*}
     and all other products equal to $0$, then by Proposition 3.1 of \cite{Burde} $\calA$ is an $LR$-algebra. Since $(e_2 e_1) e_2 = - e_3 \neq 0= e_1 (e_2 e_2)$, $\A$ is not a nearly associative algebra.
\end{exa}

\begin{lem}\label{lem:R-L} Consider a nearly associative algebra $\A$.
Then  $\A$ is an $R$-algebra if and only if it is an $L$-algebra.
\end{lem}

\begin{proof}
Let us suppose that $\A$ is an $R$-algebra. Since $\A$ is a nearly associative algebra, we obtain, for all $x,y,z \in \A$ 
\begin{equation*}
\begin{split}
 x(yz)=  (zx)y=(zy)x=y(xz).
\end{split}
 \end{equation*}
Then $\A$ is an $L$-algebra. Thus, we have shown that if $\A$ is an $R$-algebra, then it is also an $L$-algebra. As a consequence, by the opposite algebra property in Remark \ref{rmk: opposite algebra}, the converse holds as well.
\end{proof}


\smallskip

\begin{defi}
Let $A$ be a non-associative algebra. A bilinear form
$B:  A \times  A \longrightarrow  \mathbb{K}$ on $ A$ is called:
\begin{enumerate}
\item  
\textit{symmetric}   if $\displaystyle
  B(x,y)= B(y,x)$, for all $ 
x,y \in A $;
 \item     \textit{non-degenerate}   if $\displaystyle x \in A$ satisfies $\displaystyle
  B(x,y)=0$, for all $ 
 y \in A $, then $\displaystyle x=0$;
 \item     \textit{invariant} (or {\it associative})  if $\displaystyle
  B(x y,z)=B(x,y z)$,  whenever $ 
x   , y ,z \in A $.
\end{enumerate}
\end{defi}
 
\begin{defi}
A non-associative algebra $  A$ is said \textit{quadratic} if there
  exists a bilinear form $B$ on $ A$ such that
  $B$ is symmetric, non-degenerate, and invariant. It is denoted by $\displaystyle ( A,B)$, and $B$ is called an \textit{invariant scalar product} on $ A$.
\end{defi}

\noindent The following theorem proved by the authors in \cite{BBR2024} characterizes quadratic nearly associative algebras.
\begin{teo}\label{Teo characterization quad NA algebras}
    Let $(A, B_{A})$ be a quadratic algebra. Then $A$ is a nearly associative algebra if and only if $A$ is an $LR$-algebra.
\end{teo}

 \begin{pro}\label{pro:LRN} Let $(A, B_{A})$ be a quadratic algebra. Then the  three conditions are equivalent:
 \begin{enumerate}
     \item $A$ is an $R$-algebra;

     \item $A$ is an $L$-algebra;

     \item $A$ is a nearly associative algebra.
 \end{enumerate}
\end{pro}

\begin{proof}
First, we show the equivalence between Condition (1) and  (2). We assume that $A$ is an $R$-algebra. Since $B_{A}$ is invariant, we get, for all $x,y,z,t\in A$
\begin{equation*}
\begin{split}
0=B_{A}((xy)z-(xz)y,t) =B_{A}( x,y  (zt))-B_{A}( x,z  (yt))
=B_{A}( x,y  (zt)-z  (yt)).
\end{split}
 \end{equation*}
Since $B_{A}$ is non-degenerate, it follows that $y  (zt)-z  (yt)=0$ for all $ y,z,t\in A$, that is, $A$ is an $L$-algebra. Thus, we have shown that Condition (1) implies (2).
 As a consequence, by the opposite algebra property in Remark \ref{rmk: opposite algebra}, we have also that Condition (2) implies (1).

Now, let us consider a nearly associative algebra $A$. Since $(A, B_{A})$ is a quadratic algebra, by Theorem \ref{Teo characterization quad NA algebras} we infer that  $A$ is an $LR$-algebra, so it is an $R$-algebra as well as an $L$-algebra. Thus, it is proven that  Condition (3) implies (1) and (2).

Now, it is enough to show that Condition (1) implies (3). Suppose that $A$ is an $R$-algebra. We proved already that Condition (1) implies (2), so  $A$ is also an $L$-algebra, that is, $A$ is an $LR$-algebra.
By Theorem \ref{Teo characterization quad NA algebras} we conclude that  $A$ is a nearly associative algebra, completing the proof.
\end{proof}

\begin{rmk}
 The characterization of nearly associative algebras presented in Proposition \ref{pro:LRN} in terms of $R$-algebras and $L$-algebras in the context of quadratic algebras allows us to save a significant effort in analyzing the structure of the bialgebras of the nearly associative algebras.
\end{rmk}

\section{Representation of nearly associative algebras}\label{Section: Representation of NA algebras}

In this section, we introduce the concept of representation of a nearly associative algebra. We define the adjoint and coadjoint representations. We establish the conditions for the existence of the coadjoint representation, and provide a general result that applies to any non-associative algebra.

\begin{defi} \label{def: representation}
Let $(A, \cdot )$ be a non-associative algebra satisfying conditions $(P_1), \ldots , (P_n)$, $V$ be a vector space,
 and $\displaystyle  r,l:{A} \longrightarrow \mbox{End }(V)$ be a pair of linear maps.
On the vector space sum $A\oplus  V$ we consider the following
product:
\begin{equation}
\label{eq: product of A+V}
\begin{split}
(x+u)\ast (y+v):=  x \cdot y+  l(x)(v)+r(y)(u),
\end{split}
 \end{equation}
 for $   x,y \in A $ and $ u,v \in V$. We call $(V,r,l)$ an \emph{$A$-bimodule}  if $(A\oplus  V, \ast)$ is an algebra  that also satisfies the conditions $(P_1), \ldots , (P_n).$ In this case,  $(r,l)$ is  called a
\emph{representation}  of $ A$ in $V$ relative to the
properties $(P_1), \ldots , (P_n).$
\end{defi}

 Accordingly to the above definition, we define a representation of a nearly associative algebra as follows. 

 \begin{defi}
     Let $\calA$ be a nearly associative algebra, $V$ be a vector space, and $\displaystyle  r,l:{A} \longrightarrow \mbox{End }(V)$ be a pair of linear maps. Then  $\displaystyle (r,l) $ is a
 representation of $\calA$ in $V$ if $(A\oplus V, \ast)$, where $\ast$ is the product defined by \eqref{eq: product of A+V}, is a nearly associative algebra.
 \end{defi}

\begin{pro} \label{pro:representation}
Let $\calA$ be  a nearly associative algebra, $V$ be a vector space,
 and $\displaystyle l,r:{\calA} \longrightarrow \mbox{End }(V)$ be linear maps. 
The pair  $\displaystyle (r,l) $ is a
 representation  of $\calA$ in $V$  
if and only if, for all $x,y \in {\calA}$, 
\begin{eqnarray} 
&&  l(x)l(y ) =r(y) r(x), \label{eq:repre1}\\
&&  l(x)r(y)=l(yx),  \label{eq:repre2}\\
&&  r(x) l(y)=r(xy) \label{eq:repre3}.
\end{eqnarray}
\end{pro}

\begin{proof}
First, suppose that the pair  $\displaystyle (r,l) $ is a
 representation of $\calA$ in $V$. Then $(A\oplus V, \ast)$, where $\ast$ is the product defined by \eqref{eq: product of A+V}, is a nearly associative algebra. Thus,
for all $   x,y  \in \calA $ and $  v  \in V,$
\begin{equation*}
\begin{split}
l(x)     l(y )  (v)  &= x  \ast ( l(y )  (v))= x  \ast ( y  \ast v)
  = ( v\ast  x)  \ast y = ( r(x)(v)  )  \ast y= r(y) r(x)(v) ,
\end{split}
\end{equation*} 
\begin{equation*}
\begin{split}
l(x)r(y)(v) &=x  \ast (r(y)(v))= x  \ast (v \ast y) = ( y\ast  x)\ast v 
= ( y   x)\ast v =  l(yx)  (v)
\end{split}
\end{equation*}
\begin{equation*}
\begin{split}
r(x) l(y)(v) &=  l(y)(v)   \ast x=     (y\ast v)  \ast x =
v  \ast ( x \ast y)
= v  \ast ( x   y)
    =   r(xy)(v).
\end{split}
\end{equation*}
Thus the Conditions \eqref{eq:repre1}-\eqref{eq:repre3} are satisfied.
 
 Conversely, assume that Conditions \eqref{eq:repre1}-\eqref{eq:repre3} are satisfied. Then, for all $x,y,z\in \A$ and $u,v,w\in V$, we have 
 \begin{align*}
     \Big((x+u)\ast (y+v)\Big)\ast (z+w)&= \Big( x y+  l(x)(v)+r(y)(u)\Big)\ast (z+w)= (xy)z+l(xy)(w)+r(z)l(x)(v)+r(z)r(y)(u)\\
     &= y(zx)+l(y)r(x)(w)+r(zx)(v)+l(y)l(z)(u)=(y+v)\ast \Big(zx+l(z)(u)+r(x)(w)\Big)\\
     &=(y+v)\ast \Big((z+w)\ast (x+u)\Big).
 \end{align*}
 Thus, $(\A\oplus V, \ast)$ is a nearly associative algebra and $(r,l)$ is a representation of $\A$ in $V$ as required. 
\end{proof}

\begin{pro}  
Let $\calA$ be a nearly associative algebra and consider the operators of left and right multiplications $ \rmL,\rmR : \A \longrightarrow \End(\A)$. Then $(\rmR,\rmL)$ is a representation of $\calA$ in $\calA$ called the {\it adjoint representation} (or {\it regular representation}) of $\calA$  .
\end{pro}
\begin{proof}
The proof is a direct consequence of Proposition \ref{pro:representation} and Remark \ref{rmk: left and righ multiplication}.    
\end{proof}


\smallskip

Let $A$ be an algebra and $(V,r,l)$ be a $A$-bimodule. As usual, we consider $V^\ast$ the dual vector space of $V$, i.e., the vector space of all linear forms on $V$. Then the dual maps 
$\displaystyle l^\ast:A  \longrightarrow \End(V^\ast)$  and $\displaystyle r^\ast: A  \longrightarrow \End(V^\ast)$ of $l$ and $r$, respectively, are defined by (for  $ x \in A, f \in {V}^\ast$): \begin{equation}\label{eq: def dual maps}
\begin{split}
l^\ast(x)(f):=f     r(x) \, \mbox{ and } \,
r^\ast(x)(f):=f    l(x),
\end{split}
\end{equation}
respectively. Here, we also write the composition of a linear dual map of ${V}^\ast$ with an endomorphism on $V$ by juxtaposition.

\begin{rmk}
    Let $\calA$ be a nearly associative algebra, $V$ be a vector space and $(r,l)$ be a representation of $\A$ in $V$. If $l^\ast$ and $r^\ast$ are the dual map of $l$ and $r$, respectively, $(r^\ast, l^\ast)$ is not necessarily  a representation of $\A$ in $V^\ast$.
\end{rmk}

\begin{rmk}
In \cite{DS} the authors use a different definition for the dual maps in the context of nearly associative algebras. They define the dual maps
$\displaystyle l^\ast:A  \longrightarrow \End(V^\ast)$  and $\displaystyle r^\ast: A  \longrightarrow \End(V^\ast)$ of $l$ and $r$, respectively, by $l^\ast(x)(f):=f l(x) $ and $r^\ast(x)(f):=f r(x)$ for  $ x \in A, f \in {V}^\ast$, respectively.
However, in many classes of algebras, the pair  $\displaystyle (r^\ast,l^\ast) $ so defined is not a representation of $A$ in $V^\ast$. For example, given  $ ({\mathfrak g}, [-,-])$ a Lie algebra and the map
$\displaystyle \mbox{ad} :{\mathfrak g} \longrightarrow {\mathfrak gl}
({\mathfrak g})$  defined by
$\displaystyle \ad_x(y)  =  (\ad x)(y) : = 
[x,y],$
for   $x,y \in {\mathfrak g}$. If we consider $L:=\ad $ and $R:=-\ad $, then $(R,L)$ is a representation of $\mathfrak g$ in $\mathfrak g$. However, 
if  $\rmL^\ast(x)(f)=f     L(x)$  and $ 
\rmR^\ast(x)(f)=f    R(x) =-f    L(x)$ for all $ x \in {\mathfrak g}$  and $f \in {\mathfrak g}^\ast$, then $({\mathfrak g}^\ast,\rmR^\ast,\rmL^\ast)$ is not a $ {\mathfrak g} $-bimodule, except in case of $2$-nilpotent algebras.
\end{rmk}

\begin{lem}\label{lem:dual} Let $\A$ be a nearly associative algebra  and $(\rmR,\rmL)$ be its adjoint representation.
Then the pair $(\rmR^\ast,\rmL^\ast)$ is a
representation of $\A$ in $\A^\ast$ if and only if
\begin{equation}
\label{eq:coadjointcond}
\begin{split}
 \rmL_{xy} =\rmR_{y x},
\end{split}
 \end{equation}
 whenever $ x,y \in \A$. In these conditions, the representation
 $(\rmR^\ast,\rmL^\ast)$ is called the \textit{coadjoint representation} (or {\it co-regular representation}) of $\A$.
\end{lem}

\begin{proof}
Let $ x,y \in \A$ and $f \in {V}^\ast$. By Remark \ref{rmk: left and righ multiplication}, we get, 
\begin{eqnarray*} 
&&  \rmL^\ast_x \rmL^\ast_y(f)= 
 \rmL^\ast_x (f    \rmR_y )= 
f    \rmR_y      \rmR_x  
 \overset{\eqref{eq:lrm1}}{=} 
f    \rmL_x     \rmL_y   
=\rmR^\ast_y (f \rmL_x ) 
=\rmR^\ast_y \rmR^\ast_x (f) . 
\end{eqnarray*} 

More, we have $\rmL^\ast_{y x}(f) =f\rmR_{y x}$ and 
\begin{eqnarray*} 
&& \rmL^\ast_x \rmR^\ast_y(f) =
 \rmL^\ast_x  (f\rmL_y)  = 
    f\rmL_y \rmR_x  \overset{\eqref{eq:lr2}}{=} 
 f\rmL_{xy}.  \end{eqnarray*} 
Thus, $\rmL^\ast_x \rmR^\ast_y(f) =\rmL^\ast_{y x}(f)$ if and only if Condition \eqref{eq:coadjointcond} holds. 
 
 Finally, we have $\rmR^\ast_{xy}(f)=f\rmL_{xy}$ and
\begin{eqnarray*} 
&& \rmR^\ast_{x}\rmL^\ast_y(f)= \rmR^\ast_x(f\rmR_y)= 
 f\rmR_y \rmL_x \overset{\eqref{eq:lr3}}{=}  
f\rmR_{y x}.
\end{eqnarray*} 
Again, $\rmR^\ast_{x}\rmL^\ast_y(f)= \rmR^\ast_{xy}(f)$ if and only if Condition \eqref{eq:coadjointcond} holds.

Therefore, the pair $(\rmR^\ast,\rmL^\ast)$ is a
representation of $\A$ in $\A^\ast$ if and only if Condition \eqref{eq:coadjointcond} is satisfied. 
\end{proof}

\begin{rmk} 
Let  $\calA$ be a nearly associative algebra. The condition  $\rmL_{xy} =\rmR_{y x}$ for all $ x,y \in \A$ is equivalent to the
condition  $\rmL_y \rmR_x =\rmR_y \rmL_x$ for all $ x,y \in \A$, which means, $(xy)z=z(yx)$
 for all  $x,y,z\in \calA$.
\end{rmk}

 \begin{pro}\label{pro:coajoint representation} Let $\A$ be a nearly associative algebra and $(\rmR,\rmL)$ be its adjoint representation. Then the following conditions are equivalent:
 \begin{enumerate}
     \item $(\rmR^\ast,\rmL^\ast)$ is a representation of $\A$ in $\A^\ast$;

     \item $\A$ is an $L$-algebra;

     \item $\A$ is an $R$-algebra.
 \end{enumerate}
\end{pro}

\begin{proof}
Suppose first that $(\rmR^\ast,\rmL^\ast)$ is a representation of $\A$ in $\A^\ast$. Then by Lemma \ref{lem:dual} $\rmL_{xy} =\rmR_{y x}$ for all $x,y\in \A$. 
Thus it follows that $x(yz)=(zy)x=y(xz)$, for all $x,y,z\in \A$, i.e., $\A$ is an $L$-algebra.

Conversely, suppose that $\A$ is an $L$-algebra. Then
$(xy)z=(xz)y=z(yx)$, for all $x,y,z\in \A$, that is $\rmL_{xy} =\rmR_{y x}$ for all $x,y\in \A$. Thus, by Lemma \ref{lem:dual} $(\rmR^\ast,\rmL^\ast)$ is a representation of $\A$ in $\A^\ast$. As a consequence, the equivalence of Conditions (1) and (2) is proved.

Finally, the equivalence of Conditions (2) and (3) comes from Lemma \ref{lem:R-L}.
\end{proof}

\begin{rmk} 
Let  $\calA$ be a nearly associative algebra. Then the pair $(\rmR,\rmL)$ is a representation of  $\calA$  in $\calA$, and  $(\calA, \rmR,\rmL)$ is a  $\calA$-bimodule.
Applying Proposition \ref{pro:coajoint representation},  if  $\A$ is an $L$-algebra or an $R$-algebra, then the pair $( \rmR ^\ast, \rmL ^\ast)$ is a representation of $\calA$ in $\calA^\ast$, thus $(\calA^\ast, \rmR^\ast, \rmL^\ast)$ is also a  $\calA$-bimodule.
\end{rmk}

\noindent As a consequence of Theorem \ref{Teo characterization quad NA algebras} and Proposition \ref{pro:coajoint representation}, we have the following.
\begin{cor}
    Let $(\A, B_{\A})$ be a quadratic nearly associative algebra and $(\rmR,\rmL)$ be its adjoint representation. Then the pair $(\rmR^\ast,\rmL^\ast)$ is a
representation of $\A$ in $\A^\ast$.
\end{cor}

We conclude this section by proving a general result that holds for any non-associative algebra $A$ and generalize the above result.


\begin{pro} \label{pro: Quad Alg and bimodules}
 Let $A$ be an algebra,  $\rmL$ and $\rmR$ be the operators of left and right multiplication, respectively.  If $(A, B_A)$ is a quadratic algebra, then there exists an isomorphism $\Phi:A \longrightarrow A^\ast$ such that $\Phi(\rmL_x(a))=\rmL^\ast_x(\Phi(a))$ and $\Phi(\rmR_x(a))=\rmR^\ast_x(\Phi(a))$ for all $x,a \in A$.
\end{pro}

\begin{proof}
Consider the map $\Phi:  A \longrightarrow A^\ast$ defined by $\Phi(a):=B_A(a,-)$, for   $a\in A$.
Since $B_A$ is non-degenerate, the map $\Phi $   is an isomorphism of vector space. Moreover, for any $x,a,b\in A$
    \begin{equation*}
        \Phi(\rmL_x(a))(b)=B_A(xa,b)=B_A(bx,a)=B_A(a,\rmR_x(b))=\Phi(a)(\rmR_x(b))=\rmL^\ast_x(\Phi(a))(b),
    \end{equation*}
that is, $\Phi(\rmL_x(a))=\rmL^\ast_x(\Phi(a))$ for all $x,a\in A$. Similarly, it follows that $\Phi(\rmR_x(a))=\rmR^\ast_x(\Phi(a))$ for all $x,a \in A$, as required.
\end{proof}

\begin{teo}\label{Teo: quadratic and bimodules}
    Let $A$ be an algebra, and $\rmL$ and $\rmR$ be the operators of left and right multiplication, respectively. Then there exists a bilinear form $B_A:A \times A \longrightarrow \mathbb K$ such that $(A, B_A)$ is quadratic if and only if there exists an isomorphism $\Phi:A \longrightarrow A^\ast$ such that $\Phi(\rmL_x(a))=\rmL^\ast_x(\Phi(a))$ and $\Phi(\rmR_x(a))= \rmR^\ast_x(\Phi(a))$ whenever $x,a \in A$.
\end{teo}
\begin{proof}
    If $(A, B_A)$ is a quadratic algebra,  the desired conclusion follows from Proposition \ref{pro: Quad Alg and bimodules}.

    Conversely, suppose that there exists an isomorphism $\Phi:A \longrightarrow A^\ast$ such that $\Phi(\rmL_x(a))=\rmL^\ast_x(\Phi(a))$ and $\Phi(\rmR_x(a))=\rmR^\ast_x(\Phi(a))$ for all $x,a \in A$. Let us consider the map $T: A \times A \longrightarrow \mathbb K$ defined by $T(x,y):=\Phi(x)(y)$ for  $x,y\in A$. Since $\Phi$ is an isomorphism, $T$ is a non-degenerate bilinear form. Also, for all $x,y,z\in A$, we have that
    \begin{equation*}
        T(xy,z)=\Phi(xy)(z)=\Phi(\rmR_y(x))(z)=\rmR^\ast_y(\Phi(x))(z)=\Phi(x)(\rmL_y(z))=\Phi(x)(yz)=T(x,yz),
    \end{equation*}
    that is, $T$ is invariant. More, $T$ satisfies the following property: for all $x,y,z\in A$
    \begin{equation}\label{eq: T}
        T(xy,z)=T(y,zx). 
    \end{equation}
     In fact, for all $x,y,z\in A$
    \begin{equation*}
        T(xy,z)=\Phi(xy)(z)=\Phi(\rmL_x(y))(z)=\rmL^\ast_x(\Phi(y))(z)=\Phi(y)(\rmR_x(z))=\Phi(y)(zx)=T(y,zx).
    \end{equation*}
 Now, let us consider the following two new bilinear forms $T_s: A \times A \longrightarrow \mathbb K$ and $T_a: A \times A \longrightarrow \mathbb K$ defined  by
 $$ T_s(x,y):=\frac{1}{2}\Big(T(x,y)+T(y,x)\Big) \quad \mbox{ and } \quad T_a(x,y):=\frac{1}{2}\Big(T(x,y)-T(y,x)\Big),$$
    for $x,y \in A$, respectively. Clearly, $T=T_s + T_a$, with $T_s$ symmetric and $T_a$  antisymmetric (i.e., $T_a(x,y)=- T_a(y,x)$ for all $x,y \in A$). Now, since $T$ is invariant and satisfy the property \eqref{eq: T}, for all $x,y,z\in A$ we have
    $$
    T_s(xy,z)=\frac{1}{2}\Big(T(xy,z)+T(z,xy)\Big)=\frac{1}{2}\Big(T(x,yz)+T(yz,x)\Big)= T_s(x,yz),
    $$
    thus $T_s$ is a symmetric invariant bilinear form. Similarly, $T_a$ is an antisymmetric invariant bilinear form. 
   Moreover,
    for all $x,y,z\in A$
    $$
    T_a(xy,z)=\frac{1}{2}\Big(T(xy,z)-T(z,xy)\Big)=\frac{1}{2}\Big(T(y,zx)-T(zx,y)\Big)=T_a(y,zx).
    $$
    As a consequence, for all $x,y,z \in A$
    $$
    T_a(xy,z)= T_a(y,zx)
    =T_a(yz,x)=-T_a(x,yz)=-T_a(xy,z).
    $$
    Thus, since $\mathbb K$ is of characteristic different from 2, for all $x,y,z\in A$, it follows that
    \begin{equation}\label{eq: T_a(AA,A)}
        T_a(xy,z)=0.
    \end{equation}

If $N:=\{ x\in A \ | \ T_a(x,y)=0 \mbox{ for all } y\in A \}$, then by \eqref{eq: T_a(AA,A)} we have $A^2\subseteq N$, and as a consequence $N$ is a two-sided ideal of $A$.  Now, let $W:=\{ x\in A \ | \ T_s(x,y)=0 \mbox{ for all } y\in A \}$.
   Since $T$ is non-degenerate, $N\cap W= \{0\}$. 
   Then there exists a vector subspace $V\subseteq A$ such that $A=W\oplus V$. Clearly,  $N\subseteq V.$
     Moreover, since $A^2\subseteq N\subseteq V$, $V$ is also a two-sided ideal of $A$ and ${T_s}\,|_{V \times V}$ is non-degenerate.

    Consider a symmetric non-degenerate bilinear form $H: W \times W \longrightarrow \mathbb K$. Then we can define the following bilinear form $\widetilde{H}:A \times A \longrightarrow \mathbb K$ such that
   $ \widetilde{H}|_{W\times W}=H$ and $ \widetilde{H}(V,A)=\widetilde{H}(A,V)=\{0\}$.
    Clearly, $\widetilde{H}$ is symmetric. Now, we consider the bilinear form 
    $$
    B_A:=T_s + \widetilde{H}: A \times A \longrightarrow \mathbb K.$$
    Since $T_s$ and $\widetilde{H}$
    are symmetric, then $B_A$ is also symmetric. Moreover, given $w\in W $ and $v\in V$, if $B_A(w+v,x)=0$ for all $x\in A$, then $\widetilde{H}(w,x)+ T_s(v,x)=0$. Now, if $x\in W$, then $T_s(v,x)=0$. So, it follows that $\widetilde{H}(w,x)=0$, and then $H(w,x)=0$. Since $ H $ is non-degenerate, it follows $w=0$. Now, if suppose that $x\in V$, then $\widetilde{H}(w,x)=0$. Thus, it follows that $T
    _S(v,x)=0$. Since  ${T_s}\,|_{V \times V}$ is non-degenerate, it follows that $v=0$. Thus $B_A$ is non-degenerate. 
    
    Finally, since $A^2 \subseteq V$, $\widetilde{H}(V,A)=\widetilde{H}(A,V)=\{ 0\}$ and $T_s$ is invariant, we have $B_A(xy,z)=T_s(xy,z)=T_s(x,yz)=B_A(x,yz)$ for all $x,y,z\in A$.
    Then $B_A$ is invariant. Thus $(A, B_A)$ is a quadratic algebra as required.
\end{proof}

\begin{defi} \label{def: isomorphismrepresentation}
Let $A$ be a  non-associative  algebra, $(V_1,r_1,l_1)$ and $(V_2,r_2,l_2)$ be two $A$-bimodules.
A linear map $\displaystyle
\Phi :V_1\longrightarrow V_2$  is a \emph{morphism} of $A$-bimodules
if for all $x  \in A, v_1 \in V_1$,
\begin{equation}
\begin{split}
\Phi(r_1(x)(v_1))=r_2(x)(\Phi(v_1)) \mbox{ and } \Phi(l_1(x)(v_1))=l_2(x)(\Phi(v_1)).
\end{split}
 \end{equation}
More, $\displaystyle
\Phi $ is called an \emph{isomorphism} of $A$-bimodules if  it is also bijective.
\end{defi}

\noindent A consequence of the previous theorem is the following result for $A$-bimodules.

\begin{cor} \label{cor:existence of quadratic}
    Let $A$ be an algebra,  $\rmL$ and $\rmR$ be the operators of left and right multiplication, respectively. Then there exists a bilinear form $B_A:A \times A \longrightarrow \mathbb K$ such that $(A, B_A)$ is quadratic if and only if  $(A,\rmR,\rmL)$ and $( A^\ast, \rmR^\ast,\rmL^\ast)$ are isomorphic $A$-bimodules.
\end{cor}

 
\section{Nearly associative $L$-algebras}\label{Section: NAL-algebras}

In this section, we introduce the class of nearly associative $L$-algebras, which we show to be the appropriate class for presenting the notion of bialgebra in our context. Furthermore, we highlight some interesting properties of this class of algebras.

\begin{defi}
    A non-associative algebra $\A$ is called {\it nearly associative $L$-algebra} (or simply {\it $NAL$-algebra} for short) if it is a nearly associative algebra and an $L$-algebra simultaneously, i.e., the following conditions are both satisfied
    \begin{align*}
        &(xy)z=y(zx),\\
        &x(yz)=y(xz),
    \end{align*}
for all $x,y,z \in \A$.
\end{defi}

\begin{rmk}
    The class of nearly associative $L$-algebras is a proper subclass of the classes of nearly associative algebras and $L$-algebras. In fact, not all nearly associative algebras are also $L$-algebras (e.g., the algebra in Example \ref{exa:NA not LR}), and vice versa, not all $L$-algebras are nearly associative algebras (e.g., the algebra in Example \ref{exa:LR not NA}).
\end{rmk}

\begin{exa}\label{ex: NAL-algebra associaitve}
    Let $(\mathfrak g, [-,-]_{\mathfrak g})$ be the Heisenberg Lie algebra of dimension three, i.e., $\mathfrak g$ is spanned by $\{e_1,e_2,e_3\}$ such that $[e_1,e_2]_{\mathfrak g}=e_3$, $[e_1,e_3]_{\mathfrak g}=[e_2,e_3]_{\mathfrak g}=0$, and consider the symmetric bilinear map $\varphi:\mathfrak g \times \mathfrak g \longrightarrow \mathbb K$ such that 
    $$
    \varphi(e_1,e_1)=\varphi(e_1,e_2)=\varphi(e_2,e_2)=1, \quad
    \varphi(e_1,e_3)=\varphi(e_2,e_3)=\varphi(e_3,e_3)=0.
    $$
   Then by \cite[Example 3.15]{BBR2024} the vector space $\A:=\mathfrak g \oplus \mathbb{K}$ endowed with the product $\cdot:\A \times \A \longrightarrow \A$ defined by
   \begin{equation*} \label{eq: product of g+V}
        (x+v) \cdot (y+w):= [x,y]_{\mathfrak g} + \varphi(x,y),
    \end{equation*}
    for   $x,y\in \mathfrak g$, $v,w \in \mathbb{K}$, is a four-dimensional nearly associative $L$-algebra that is in particular associative.
\end{exa}

\begin{rmk}
By \cite[Corollary 3.11]{BBR2024} any associative $LR$-algebra is a nearly associative $L$-algebra. However, not all nearly associative algebras that are associative are also $LR$-algebras (cf. \cite[Example 8.24]{BBR2024}).
\end{rmk}

Now, we give an example of a six-dimensional nearly associative $L$-algebra that is not associative.

\begin{exa}\label{ex: NAL-algebra}
    Let $ \A$ be a vector space of dimension six with basis $\{e_1,e_2,e_3,e_4,e_5,e_6 \} $. 
     If we equipped $ \A$ with the linear product defined by
     \begin{align*}
          &e_1 e_2=e_3, \quad e_2 e_1 =e_4, \quad e_2 e_2 =e _5, \quad e_1 e_5=e_5 e_1 =  e_2 e_3 = e_4 e_2 = e_6, \quad e_2 e_4 = e_3 e_2 = 2 e_6, 
    \end{align*}
     and all other products equal to $0$, then  $ \A$ is a nearly associative algebra which is not associative (cf. \cite[Example 5.2]{BBR2024}). Also, $ \A$ is an $L$-algebra. In fact,  by easy calculation we get that $e_i (e_j e_k)=0$  if either one of the indices $ i,j,k $ is not in $\{1,2 \}$ or two of the indices $ i,j,k $ are equal to $1$, and $e_1 ( e_2 e_2)= e_6 = e_2 (e_1 e_2).$
    Thus $ \A$ is a nearly associative $L$-algebra.
\end{exa}

\begin{rmk}\label{rmk: NAL-alg is R-alg}
 If $\A$ is a nearly associative $L$-algebra, then by Lemma  \ref{lem:R-L} $\A$ is an $R$-algebra, and as a consequence $\A$ is $LR$-algebra, i.e., $(xy)z=(xz)y$ for all $x,y,z\in \A$.
\end{rmk}

\begin{pro}\label{pro:LR cond NAL}
    If $\A$ is a nearly associative $L$-algebra, then 
     \begin{align}
           & \rmL_x   \rmL_y=\rmL_y \rmL_x= \rmR_x \rmR_y=\rmR_y   \rmR_x, \label{LR1}\\
&\rmL_y   \rmR_x=  \rmL_{x  y}=\rmR_{yx} =\rmR_y \rmL_x \label{LR2},
     \end{align}
for all $x,y\in \A$.
 \end{pro}
\begin{proof}
By the above remark $\A$ is an $LR$-algebra, then $\rmL_x   \rmL_y=\rmL_y \rmL_x$ and $\rmR_x \rmR_y=\rmR_y   \rmR_x$ for all $x,y \in \A$. Also, $\A$ is in particular a nearly associative algebra, then by Condition \eqref{eq:lrm1} we have that $\rmL_y \rmL_x= \rmR_x \rmR_y$ for all $x,y\in \A$. Thus, Conditions \eqref{LR1} are satisfied.

Now, we prove Conditions \eqref{LR2}. For all $x,y,z\in \A$,
\begin{equation*}
    \rmL_{xy}(z)=(xy)z=y(zx)=z(yx)=\rmR_{yx}(z).
\end{equation*}
Also, since $\A$ is a nearly associative algebra, from Conditions \eqref{eq:lr2} and \eqref{eq:lr3} we have that $\rmL_{xy}= \rmL_y   \rmR_x$ and $\rmR_{yx}=\rmR_y \rmL_x $ for all $x,y\in \A$. Thus Conditions \eqref{LR2} hold.
\end{proof}

\begin{rmk}\label{rmk: NAL quasi-comm}
As a direct consequence of Proposition \ref{pro:LR cond NAL}, we have that if $\A$ is a nearly associative $L$-algebra, then
      \begin{equation}
       (xy)z=z(yx), \label{eq: quasi-comm}
    \end{equation}
     for all $x,y,z\in \A$.
\end{rmk}

The following proposition, concerning also the products defined in \eqref{eq:A-A+}, shows some direct consequences of the definition of nearly associative $L$-algebras.

\begin{pro}
\label{pro:cons of NAL}
    Let $\A$ be a nearly associative $L$-algebra. Then:
    \begin{enumerate}
        \item $\A \bullet \A \subseteq Z(\A^-)$;
        \item $[\A,\A]\subseteq \ann(\A^+)$;
        \item   $\displaystyle\asso^+(x,y,z)= \frac{1}{2} \asso(x,y,z)=[y,z x]$, for all $x,y,z\in \A$.
    \end{enumerate}
\end{pro}
\begin{proof}
    Let $x,y,z\in \A$. Then by Remark \ref{rmk: NAL quasi-comm} we have
    \begin{align*}
        [x  \bullet y, z]&\overset{\eqref{eq:A-A+}}{=} \frac{1}{4}\Big( (xy)z +(yx)z - z(xy) - z(yx) \Big)\overset{\eqref{eq: quasi-comm}}{=}\frac{1}{4}\Big( (xy)z +(yx)z - (yx)z - (xy)z \Big)=0,
    \end{align*}
    and
     \begin{align*}
        [x, y]\bullet z&\overset{\eqref{eq:A-A+}}{=} \frac{1}{4}\Big( (xy)z-(yx)z +z(xy)-z(yx)\Big)\overset{\eqref{eq: quasi-comm}}{=}\frac{1}{4}\Big(  (xy)z-(yx)z +(yx)z -(xy)z \Big)=0.
    \end{align*}
    Thus Conditions (1) and (2) are satisfied. Moreover, by  Remark \ref{rmk: NAL quasi-comm} and Proposition \ref{pro: consequnces def NA}
    \begin{align*}
\asso^+(x,y,z) &=(x\bullet y)\bullet z-x\bullet(y\bullet z) 
\overset{\eqref{eq:A-A+}}{=} \frac{1}{4}\Big(( x y) z+(y x) z+
z  ( x y)+z (y x)
-x( y z)-x (z y)
-( y z) x-(z y)  x \Big)\\ 
&=\frac{1}{4}\Big(( x y) z -x( y z) +z  (y x)
-(z y)  x \Big)\overset{\eqref{eq: quasi-comm}}{=} \frac{1}{2} \Big(( x y) z -x( y z) \Big)=\frac{1}{2}\asso(x,y,z)\overset{\eqref{eq:condit1}}{=}[y,zx].
\nonumber
\end{align*}
\end{proof}

As a direct consequence of (2) of Proposition \ref{very} and (1) of Proposition \ref{pro:cons of NAL} we get the following. 
\begin{cor}
    Let $\A$ be a nearly associative $L$-algebra. Then:
    \begin{enumerate}
        \item $\asso^+(x\bullet t, y,z)= \asso^+(x,y \bullet t, z)=\asso^+(x,y,z\bullet t)=0$, for all $x,y,z,t\in \A$;

        \item $(\A \bullet \A, \bullet)$ is an associative subalgebra of $\A^+$.
    \end{enumerate}   
\end{cor}

As a consequence of Proposition \ref{pro: A flexible and center condition} and Condition (1) of Proposition \ref{pro:cons of NAL} we have the following.
\begin{cor}
    Any nearly associative $L$-algebra is flexible.
\end{cor}

As a result of Proposition \ref{pro: A^+ Jordan A^- Lie} and Conditions (1) and (2) of Proposition \ref{pro:cons of NAL}, we derive the following.
\begin{cor}\label{cor: NAL Lie-poisson-Jordan}
    Any nearly associative $L$-algebra is a Lie-Poisson-Jordan admissible algebra.
\end{cor}

Recall that a Lie algebra $  (\mathfrak{g}, [-,-])$ is $3$-nilpotent if $[[[x,y],z],t]=0$ for all $x,y,z,t\in \mathfrak{g}$. Then by (2) of Proposition \ref{very}, (1) of Proposition \ref{pro:cons of NAL} and Corollary \ref{cor: NAL Lie-poisson-Jordan} we get the following.
\begin{cor}
    If $\A$ is a nearly associative $L$-algebra, then $\A^-$ is a $3$-nilpotent Lie algebra.
\end{cor}

Recall that a Lie algebra $  (\mathfrak{g}, [-,-])$ is $2$-nilpotent if $[[x,y],z]=0$ for all $x,y,z\in \mathfrak{g}$. Then we obtain the following as a direct consequence of Condition (3) of Proposition \ref{pro:cons of NAL} and Corollary \ref{cor: NAL Lie-poisson-Jordan}.
\begin{cor}
    Let $\A$ be a nearly associative $L$-algebra. Then the following conditions are equivalent:
    \begin{enumerate}
        \item $\A$ is an associative algebra;
        \item $\A^+$ is a commutative associative algebra;
        \item $\A^-$ is a $2$-nilpotent Lie algebra.
    \end{enumerate}
\end{cor}

Next, we present another property of the underlying Lie algebra $\A^-$ of a nearly associative $L$-algebra $\A$. To this end, we recall the definition of a central extension of a Lie algebra.
Let $(\mathfrak{g}, [-,-]_{\mathfrak{g}})$ be a Lie algebra, $V$ be a vector space, and $\omega: \mathfrak{g} \times \mathfrak{g} \longrightarrow V$ a bilinear map. Then the vector space $\mathfrak{h}:=\mathfrak{g}\oplus V$ endowed with the product defined by 
$$
[x+v ,y+w]_{\mathfrak{h}}:=[x,y]_{\mathfrak{g}}+ \omega(x,y),
$$
for $x,y\in \mathfrak{g}$ and $v,w\in V$, is a Lie algebra if and only if $\omega$ is a 2-cocycle of $\mathfrak{g}$, i.e., for $x,y,z\in \mathfrak{g}$,
\begin{equation*}
    \omega(x,y)=-\omega(y,x),
\end{equation*}
\begin{equation*}
\omega([x,y]_{\mathfrak{g}},z)+\omega([y,z]_{\mathfrak{g}},x)+\omega([z,x]_{\mathfrak{g}},y)=0.
\end{equation*}
The Lie algebra $\mathfrak{h}$ is called the {\it central extension} of $\mathfrak{g}$ by $V$ by means of $\omega$.
Remark that $\mathfrak{h}$ is called ``central" because $V\subseteq Z(\mathfrak{h}).$
Also, it is referred to as an ``extension”  since it fits into the short exact sequence $V \hookrightarrow \mathfrak{h} \twoheadrightarrow \mathfrak{g}.$ 
Indeed, the cocycle $\omega$ arises from the canonical projection $\pi: \mathfrak{h} \longrightarrow V$ via the relation $ \omega(x, y) = \pi([x, y]).$

\begin{pro}\label{pro: A^- central extension}
     Let $\A$ be a nearly associative $L$-algebra. Then $(\A^-,[-,-])$ is the central extension of a Lie algebra $(\mathfrak{g},[-,-]_{\mathfrak{g}})$ by means of a $2$-cocycle $\omega$.
\end{pro}
\begin{proof}
   There exists a vector subspace $\mathfrak{g}\subseteq \A^-$ such that $\A^-=\mathfrak{g}\oplus (\A \bullet \A)$ as vector spaces.  Since by (1) of Proposition \ref{pro:cons of NAL} $\A \bullet \A \subseteq Z(\A^-)$, we may have for all $x,y \in \mathfrak{g}$ and $a,b\in \A \bullet \A$
   $$
   [x+a, y+b]:= [x,y]_{\mathfrak{g}}+ \omega(x,y),
   $$
   where $
    [x,y]_{\mathfrak{g}} \in 
   \mathfrak{g}$ and $
    \omega(x,y) \in  \A \bullet \A 
   $.  
   We claim that $(\mathfrak{g}, [-,-]_{\mathfrak{g}})$ is a Lie algebra and $\omega:\mathfrak{g}\times \mathfrak{g}\longrightarrow \A \bullet \A$ is a $2$-cocycle, and this will complete the proof.
   Since $(\A^-,[-,-])$ is a Lie algebra,
   for $x,y \in\mathfrak{g}$ and $a,b\in \A \bullet \A$
   \begin{align*}
   ¨[x+a,y+b]=-[y+b,x+a] \; \Longrightarrow [x,y]_{\mathfrak{g}}+ \omega(x,y)=-[y,x]_{\mathfrak{g}}- \omega(y,x),
   \end{align*}
   and consequently $[x,y]_{\mathfrak{g}}= -[y,x]_{\mathfrak{g}}$ and $\omega(x,y)=- \omega(y,x)$. Also, for $x,y,z \in \mathfrak{g}$ and $a,b,c\in \A \bullet \A$, from
   \begin{align*}
   ¨ [x+a,[y+b,z+c]]+[y+b,[z+c,x+a]] +[z+c,[x+a,y+b]]=0, 
    \end{align*}
    we get
   \begin{align*}
   [x,[y,z]_{\mathfrak{g}}]_{\mathfrak{g}}+  [y,[z,x]_{\mathfrak{g}}]_{\mathfrak{g}}+  [z,[x,y]_{\mathfrak{g}}]_{\mathfrak{g}}+ \omega(x,[y,z]_{\mathfrak{g}}) + \omega(y,[z,x]_{\mathfrak{g}}) + \omega(z,[x,y]_{\mathfrak{g}})=0.
   \end{align*}
   Thus $(\mathfrak{g}, [-,-]_{\mathfrak{g}})$ satisfies the Jacobi identity and $\omega([x,y]_{\mathfrak{g}},z)+\omega([y,z]_{\mathfrak{g}},x)+\omega([z,x]_{\mathfrak{g}},y)=0$, as claimed.
\end{proof}

\begin{rmk}
     The Lie algebra $\mathfrak{g}$ in the Proposition \ref{pro: A^- central extension} is isomorphic to $\A^-/(\A \bullet \A)$.
\end{rmk}

We conclude this section with a characterization of nearly associative $L$-algebras.

\begin{pro}\label{pro: characterization NAL}
An algebra $\A$ is a nearly associative $L$-algebra if and only if the following four conditions hold:
\begin{enumerate}
\item  $\A^-$ is a Lie algebra;
\item $\A \bullet \A \subseteq Z(\A^-)$;
\item $[\A,\A]\subseteq \ann(\A^+)$;
\item    $ \asso^+(x,y,z) =[y,zx]$, for  all $x,y,z\in \A$. 
\end{enumerate}
\end{pro}
\begin{proof}
    By Propositions~\ref{pro: A^+ Jordan A^- Lie} and \ref{pro:cons of NAL}, it is clear that if $\A$ is a $NAL$-algebra then the four conditions above are satisfied.

Conversely, suppose that $\A$ satisfies these four conditions. First notice that from Condition (2) it follows that $[y,xz]=-[y,zx]$  for all $x,y,z\in \A$, then
\begin{equation*}
    [y,[z,x]]
=\frac{1}{2}\Big( [y,zx]-[y,xz]\Big)=[y,zx].
\end{equation*}
Thus, by Condition (4), we have for $x,y,z\in \A$ 
\begin{equation}\label{eq:1ab}
    \asso^+(x,y,z)=[y,[z,x]], 
\end{equation}
which ensures the second condition of Proposition \ref{very}.
Moreover, from Conditions (2) and (3) we have that  $ [x,y]\bullet z+ [ x\bullet y , z]+ [x,z]\bullet y  +[  x\bullet z ,y]= 0, $ for all $x,y,z\in \A$, that guarantees the third condition of Proposition \ref{very}. 
Also, by Condition (1) $\A^-$ is a Lie algebra. Thus, the three conditions of Proposition \ref{very} are satisfied and, as a consequence, $A$ is a nearly associative algebra.

It remains to prove that $\A$ is also an $L$-algebra. From Conditions (2) and (3) and using \eqref{eq:cdot},
for every $x, y , z \in \A $ we have
\begin{align*}
x (y z)&= x[y, z]+ x (y\bullet z) = [x, [y, z]]+x\bullet [y, z] +[x,  y\bullet z ] + x\bullet (y\bullet z)= [x, [y, z]]+ x\bullet (y\bullet z) ,
\end{align*}
and, by switching $x$ with $y$, we also have $y (x z) = [y, [x, z]]+ y\bullet (x\bullet z)$.
Since $[-,-]$ is anti-commutative and $\bullet$ is commutative, using Condition  \eqref{eq:1ab}, we  have
\begin{align*}
    x (y z)-y (x z)&=[x, [y, z]]+ x\bullet (y\bullet z)- [y, [x, z]]- y\bullet (x\bullet z)=[x, [y, z]]+ [y, [z, x]]+ (y\bullet z)\bullet x - y\bullet (z\bullet x)\\
    &=[x, [y, z]]+ [y, [z, x]]+\asso^+(y,z,x) \overset{\eqref{eq:1ab}}{=}[x, [y, z]]+ [y, [z, x]] + [z,[x,y]].
\end{align*}
 Therefore, since $[-,-]$ is a Lie structure on
 $\A$, we get $x (y z)-y (x z)= 0$ for all $x, y , z \in \A $, i.e., $\A$ is an $L$-algebra. We conclude that $\A$ is an $NAL$-algebra as required.
\end{proof}

\section{Nearly coassociative coalgebras, $L$-coalgebras and $R$-coalgebras}\label{Section: Coalgebras}

This section aims to introduce the notions of nearly coassociative coalgebras, $L$-coalgebras and $R$-coalgebras and present some of their properties. 

Let $(A,\cdot)$ be an algebra and  $ \Delta_A  : A  \longrightarrow  A  \otimes A $  be a linear map called \textit{comultiplication}. Then the pair $(A, \Delta_A)$ is called a \textit{coalgebra}. In the following, we use Sweedler's notation, so we write  $\Delta_A (x)=\sum_{(x)}x_{(1)}\otimes x_{(2)}$,
where $x_{(1)}$ and $ x_{(2)}$ are elements in $A$.

In the following, denote by  $\displaystyle \tau:A\otimes A\longrightarrow A\otimes A$ the linear map (the so-called  \textit{twist map}) given by $\tau(x\otimes y):= y\otimes x$, for  $\displaystyle x,y \in A $, by 
$\displaystyle \xi:A\otimes A\otimes A\longrightarrow A\otimes A\otimes A$ the linear cycle map defined by $\xi(x\otimes y\otimes z):= y\otimes z \otimes x$, for  $\displaystyle x,y,z \in A$, and by $\mbox{I}: A\longrightarrow A$ the identity map on $ A$.

Let us denote by $\displaystyle m: A\otimes A\longrightarrow A$   the bilinear product on $A$, that is,  $m(x\otimes y):= x\cdot y,\,$  for   $\displaystyle x,y \in A$, i.e., $(A,m =\cdot)$. Then the condition in the definition of nearly associative algebra may be written as:
\begin{equation*}
\begin{split}
m  (m\otimes \mbox{I})&=m (\mbox{I}\otimes m) \xi.
\end{split}
\end{equation*}

\begin{defi}\label{def: nearly coassociative coalgebra} Let $\calA$ be a vector space provided with a comultiplication $\displaystyle \Delta_\A:\calA\longrightarrow \calA\otimes \calA $. We say that $\calA$ is a \textit{nearly coassociative coalgebra} if
\begin{equation*}
\begin{split}
(\Delta_\A \otimes \mbox{I}) \Delta_\A &=\xi (\mbox{I} \otimes \Delta_\A) \Delta_\A    .
\end{split}
\end{equation*}
\end{defi}

\noindent
Similarly, the condition of $L$-algebra may be written as:
\begin{equation*}
\begin{split}
m (\mbox{I}\otimes m)  &=m (\mbox{I}\otimes m) (\tau \otimes \mbox{I}),
\end{split}
\end{equation*}
and the condition of $R$-algebra may be written as:
\begin{equation*}
\begin{split}
m (m \otimes\mbox{I})  &=m (m \otimes\mbox{I}) ( \mbox{I} \otimes \tau).
\end{split}
\end{equation*}

\begin{defi}\label{def: LR coalgebras} Let $A$ be a vector space provided with a comultiplication $\displaystyle \Delta_A :A\longrightarrow A\otimes A $. We say that the pair $(A,\Delta_A)$ is an:
\begin{itemize}
    \item \textit{$L$-coalgebra} if
\begin{equation}
\label{eq: $L$-coalgebra}
\begin{split}
 (\mbox{I}\otimes \Delta_A) \Delta_A &=(\tau \otimes \mbox{I})( \mbox{I} \otimes \Delta_A)\Delta_A.
\end{split}
\end{equation}
\item \textit{$R$-coalgebra} if
\begin{equation}
\label{eq: $R$-coalgebra}
\begin{split}
 (\Delta_A \otimes\mbox{I}) \Delta_A &=(\mbox{I}\otimes \tau)( \Delta_A \otimes\mbox{I})\Delta_A.
\end{split}
\end{equation}
\item  \textit{$LR$-coalgebra} if it is an $L$-coalgebra and an $R$-coalgebra simultaneously, i.e., Conditions \eqref{eq: $L$-coalgebra} and \eqref{eq: $R$-coalgebra} are both satisfied. 
\end{itemize}
\end{defi}

Let $(A, \cdot )$ be a non-associative algebra, and assume a non-associative algebra structure in the dual vector space $(A^*,\star)$. In the following,  we denote $f(x)$ by either $\langle
 f,x\rangle\,$ or $\langle x,f\rangle$, for $x\in
 A$ and $f\in A^\ast$. Moreover, we identify the finite dimensional algebra $ A $ and its bi-dual space $ A^{**}=( A^ *)^*  $ by the isomorphism of vector space 
 $A \longrightarrow ( A^ *)^*$ given by $ x \mapsto \langle x,-\rangle.$
 
There is a natural way to define a coalgebra structure on $A$ from the algebra structure of its dual $A^*$.
From the multiplication 
$\displaystyle  \star   : A^\ast \otimes A^\ast   \longrightarrow A^\ast  $ on $\displaystyle   A^\ast $ we infer a comultiplication $\displaystyle  \Delta_A := \ts   : A   \longrightarrow A \otimes A $  on $A$ (using the identification between $ A $ and the bi-dual $ A^{**}=( A^ *)^* $)  
given by the following relation:
\begin{equation}\label{eq: product A^*=coproduct A}
\begin{split}
\langle \Delta_A  (x ), f \otimes g\rangle &= \langle x , f \star g \rangle , 
\end{split}
\end{equation}
for  $ x  \in A$ and $ f,g \in A^\ast$.  On the other way around, we may also endow the dual of a coalgebra with an algebra structure. Given a coalgebra  $(A, \Delta_A)$, then $(A^\ast, \star)$ is an algebra where the multiplication $\star: A^\ast \otimes A^\ast \longrightarrow A^\ast$ comes from \eqref{eq: product A^*=coproduct A}.


Similarly, we can also consider a coalgebra structure on the dual $A^*$ from the algebra structure of $A$,  by taking into account the 
comultiplication $\displaystyle  \Delta_{A^\ast}   : A^\ast   \longrightarrow A^\ast \otimes A^\ast $  on $A^\ast$ given by 
\begin{equation*}
\begin{split}
\langle \Delta_{A^\ast}  (f ), x \otimes y\rangle &:= \langle f ,x \cdot y \rangle , 
\end{split}
\end{equation*}
for $ x,y  \in A$ and $ f \in A^\ast$.

Actually,  we can establish the next result.

\begin{pro} \label{pro:  nearly coasso coalg and L-coalg}
Let $(A, \Delta_A)$ be a coalgebra and $(A^*, \star)$ be an algebra, where 
$\Delta_A $ and $  \star$ are connected by the relation \eqref{eq: product A^*=coproduct A}. Then we have the following assertions.
\begin{enumerate}
    \item $(A, \Delta_A)$ is a nearly coassociative coalgebra if and only if $(A^*, \star)$ is a nearly associative algebra.

    \item  $(A, \Delta_A)$ is an $L$-coalgebra if and only if $(A^*, \star)$ is an $L$-algebra.

    \item  $(A, \Delta_A)$ is an $R$-coalgebra if and only if $(A^*, \star)$ is an $R$-algebra.
\end{enumerate}
\end{pro}
\begin{proof}
    It follows from straightforward calculations.
\end{proof}

As a consequence of Lemma \ref{lem:R-L} and Proposition \ref{pro:  nearly coasso coalg and L-coalg} we have the following.

\begin{cor}\label{cor: NCAcoal L-coal and R-coal}
    Let $(\A, \Delta_\A)$ be a nearly coassociative coalgebra. Then $(\A, \Delta_\A)$ is an $L$-coalgebra if and only if it is an $R$-coalgebra.
\end{cor}

\begin{defi}
    Let $\A$ be a vector space provided with a comultiplication $\displaystyle \Delta_{\A} : \A\longrightarrow \A\otimes \A $. We say that $\A$ is a \textit{nearly coassociative $L$-coalgebra} if it is a nearly coassociative coalgebra and an $L$-coalgebra simultaneously, i.e., the following conditions are both satisfied
\begin{align*}
(\Delta_{\A} \otimes \mbox{I}) \Delta_{\A} &=\xi (\mbox{I} \otimes \Delta_{\A}) \Delta_{\A},\\
 (\mbox{I}\otimes \Delta_{\A}) \Delta_{\A} &=(\tau \otimes \mbox{I})( \mbox{I} \otimes \Delta_{\A})\Delta_{\A}.
\end{align*}
\end{defi}

\begin{rmk}
If $(\A, \Delta_\A)$ is a nearly coassociative $L$-coalgebra, then as a consequence of Corollary \ref{cor: NCAcoal L-coal and R-coal} it follows that $\A$ is also an $R$-coalgebra.
\end{rmk}

\begin{exa}\label{ex: nearly coassociative L-coalgebra asso}
    Let $ \A$ be a vector space of dimension four with basis $\{e_1,e_2,e_3,e_4 \} $ and consider the comultiplication $\Delta_\A:\A \longrightarrow \A \otimes \A$ such that the only non-zero values in the basis are
\begin{equation*}
\Delta_\A( e_1):= 2 (e_4 \otimes e_4), \quad
\Delta_\A(e_2):= e_4 \otimes e_3- e_3 \otimes e_4 + 2 (e_4 \otimes e_4).
\end{equation*} 
Then    $ (\A, \Delta_\A)$ is a nearly coassociative $L$-coalgebra.
\end{exa}

\begin{exa}\label{ex: nearly coassociative L-coalgebra not asso}
    Let $ \A$ be a vector space of dimension six with basis $\{e_1,e_2,e_3,e_4,e_5,e_6 \} $ and consider the comultiplication $\Delta_\A:\A \longrightarrow \A \otimes \A$ such that the only non-zero value in the basis is 
\begin{equation*}
\Delta_\A(e_1):=e_6 \otimes e_6.
\end{equation*} 
     Then  $ (\A, \Delta_\A)$ is a nearly coassociative $L$-coalgebra.
\end{exa}

\section{Nearly associative bialgebras}\label{Section: NA bialgebras}

This section aims to introduce the notion of nearly associative bialgebra and give a characterization. Moreover, we show that it is equivalent to the concept of $R$-bialgebra, $L$-bialgebra, and then $LR$-bialgebra.

Drinfeld introduced the notion of Lie bialgebra in \cite{Dri1983} to study solutions to the classical Yang-Baxter equation.
After that, Zhelyabin generalized this construction to an arbitrary variety \cite{Zhe1997}. In the following, we use this construction to define the class of nearly associative bialgebras.

Let $(A, \cdot )$ be a non-associative algebra, and assume a non-associative algebra structure in the dual vector space $(A^*,\star)$.
On the vector space $\rmD(A):=A\oplus A^*$, we consider the following product:
\begin{equation}
(x+f)\diamond (y+g) :=
\underbrace{x \cdot y+\rmL_f ^*(y)+\rmR_g^*(x)}_{\in \;  A }+
\underbrace{ f\star g +\rmL_x^*(g)+\rmR_y^*(f)}_{  \in \; A^*}, 
\end{equation}
for  $  x,y\in  A ,f,g\in  A ^*$.
Remark that the elements $\rmL_f ^*(y)$ and $\rmR_g^*(x)$  belong to $A$, using the identification between $ A $ and $ ( A^ *)^*  $. Indeed, for $h \in  A^*    $ we have
$$\rmL_f ^*(y)(h)=\langle y, \rmR_f(h)\rangle=\langle y, h\star f\rangle \; \mbox{ and } \; \rmR_g^*(x)(h)=\langle x, \rmL_g(h)\rangle=\langle x, g \star h\rangle.$$
Thus $(\rmD( A ),\diamond)$ is an algebra, and both $(A, \cdot )$ and $({A}^*,\star)$ are its subalgebras.

Now, let us consider 
$B: \rmD(A) \times \rmD(A) \longrightarrow  \mathbb{K}$ given by 
\begin{equation}\label{eq: Bilinear form on the double}
 B (x+f,y+g):=\langle f,y\rangle+ \langle g, x\rangle,
\end{equation}
for $  x,y\in  A ,f,g\in  A ^*$.
It is easy to see that $B$ is a symmetric  non-degenerate  bilinear form such that 
\begin{equation*}
 B \Big((x+f)\diamond (y+g), z+l\Big)=  B \Big( x+f , (y+g)\diamond ( z+l)\Big),
\end{equation*}
whenever $  x,y,z \in  A , f,g,l\in  A ^*$. Hence $(\rmD( A ),\diamond, B)$ is a quadratic algebra.

Now, we recall the definition of bialgebras for an arbitrary variety presented by Zhelyabin in \cite{Zhe1997}.
\begin{defi} \label{def:bialgebras}
Suppose that $M$ is a variety of algebras, and $(A, \cdot )$ and $(A^*,\star)$ are algebras in $M$, where $ A^* $ is the dual vector space of $ A $. Then $(A, \cdot , \Delta_A  )$, where $\Delta_A=\ts$,  is called an \emph{$M$-bialgebra}  (or a \emph{bialgebra of type $M$})  
if the algebra $(\rmD(A), \diamond )$ belongs to
$M$.
\end{defi}

In the context of quadratic algebras, by Theorem \ref{Teo characterization quad NA algebras}, the notions of nearly associative algebra and $LR$-algebra are equivalent. Thus, since $(\rmD( A ),\diamond, B)$, where $B$ is the the bilinear form defined by \eqref{eq: Bilinear form on the double}, is a quadratic algebra such that  $(A, \cdot )$ and $(A^*,\star)$ are both subalgebras of $(\rmD( A ),\diamond)$, the notion of nearly associative bialgebra only makes sense if $(A, \cdot )$ and $(A^*,\star)$ are both nearly associative algebras and $LR$-algebras. 
Moreover, by Lemma \ref{lem:R-L}, the notions of nearly associative $L$-algebra and nearly associative $R$-algebra are equivalent. Then the notion of nearly associative bialgebra makes sense if $(A, \cdot )$ and $(A^*,\star)$ are nearly associative $L$-algebras ($NAL$-algebras), or equivalently, nearly associative $R$-algebras.
Let us write the definition of bialgebra that we will use specifically for the case of nearly associative algebras. 
  
 \begin{defi} \label{def: nearly associative bialgebras}
Let $(\A, \cdot )$ and $(\A^*,\star)$ be nearly associative $L$-algebras, where $ \A^* $ is the dual vector space of $ \A $. Then $(\A, \cdot , \Delta_\A  )$, where $\Delta_\A=\ts$,  is called \emph{nearly associative bialgebra}
if $(\rmD(\A), \diamond )$ is a nearly associative algebra.
\end{defi}

\begin{rmk} \label{rmk: nearly asso bial, $L$-bialg and $R$-bialg}
   Since $(\rmD( \A ),\diamond, B)$ is a quadratic algebra, where the bilinear form $B$ is defined by \eqref{eq: Bilinear form on the double}, by Theorem \ref{Teo characterization quad NA algebras}  $(\rmD(\A), \diamond )$ is a nearly associative algebra if and only if  $(\rmD(\A), \diamond )$ is an $L$-algebra if and only if $(\rmD(\A), \diamond )$ is an $R$-algebra. So, the concepts of the nearly associative bialgebra, $R$-bialgebra, $L$-bialgebra, and then $LR$-bialgebra are all equivalent.
\end{rmk}

\begin{teo} \label{teo: characterization nearly associative bialg}
    Let $(\A, \cdot)$ and$ (\A^*, \star)$ be two nearly associative $L$-algebras (where $ \A^* $ is the dual vector space of $ \A $).  Then $(\A, \cdot , \Delta_\A  )$, where $\Delta_\A=\ts$,  is a nearly associative bialgebra if and only if the following conditions hold: for all $x,y\in \A$
    \begin{align} \label{eq: 1 nearly asso bialgebra}
        \Delta_{\A} (x \cdot y) = \; (\mbox{I} \otimes \rmR_y) \Delta_\A(x)+ \tau(\mbox{I} \otimes \rmR_x)\Delta_\A (y) =\; \tau(\rmL_y \otimes \mbox{ I}) \Delta_\A (x) + (\rmL_x \otimes \mbox{ I}) \Delta_\A (y),
    \end{align}
    \begin{align} \label{eq: 2 nearly asso bialgebra}
        ( &\rmR_y \otimes \mbox{ I}) \Delta_\A(x) + (\mbox{I} \otimes \rmL_x) \Delta_\A(y)= \; ( \mbox{I} \otimes \rmL_y) \Delta_\A(x)+ (\rmR_x \otimes \mbox{ I}) \Delta_\A (y) = \; \tau ( \mbox{I} \otimes \rmL_y) \Delta_\A(x) +\tau (\rmR_x \otimes \mbox{ I}) \Delta_\A (y).
    \end{align}
\end{teo}
\begin{proof}
Let $x,y\in \A$ and $f,g\in \A^*$. We claim that
 \begin{align}
        (x \diamond y) \diamond f= (x\diamond f) \diamond y \; &\Longleftrightarrow\;  \Delta_\A (x \cdot y) = (\mbox{I} \otimes \rmR_y) \Delta_\A(x)+ \tau(\mbox{I} \otimes \rmR_x)\Delta_\A (y), \label{eq: equivalence 1}\\
        (f \diamond x) \diamond y= (f\diamond y) \diamond x \; &\Longleftrightarrow\;  ( \rmR_y \otimes \mbox{ I}) \Delta_\A(x) + (\mbox{I} \otimes \rmL_x) \Delta_\A(y)= ( \mbox{I} \otimes \rmL_y) \Delta_\A(x)+ (\rmR_x \otimes \mbox{ I}) \Delta_\A (y), \label{eq: equivalence 2}\\
        (f \diamond g) \diamond x= (f\diamond x) \diamond g \; &\Longleftrightarrow\; \Delta_\A (x \cdot y) = \tau(\rmL_y \otimes \mbox{ I}) \Delta_\A (x) + (\rmL_x \otimes \mbox{ I}) \Delta_\A (y),\\
        (x \diamond f) \diamond g= (x\diamond g) \diamond f \; &\Longleftrightarrow\; ( \mbox{I} \otimes \rmL_y) \Delta_\A(x)+ (\rmR_x \otimes \mbox{ I}) \Delta_\A (y) =  \tau ( \mbox{I} \otimes \rmL_y) \Delta_\A(x) + \tau(\rmR_x \otimes \mbox{ I}) \Delta_\A (y).
    \end{align}
    This will prove that $(\rmD(\A), \diamond )$ is an $R$-algebra if and only if Conditions \eqref{eq: 1 nearly asso bialgebra} and \eqref{eq: 2 nearly asso bialgebra} are satisfied, and by Definition \ref{def: nearly associative bialgebras} and Remark \ref{rmk: nearly asso bial, $L$-bialg and $R$-bialg} the proof is complete.

    First, we show the Equivalence \eqref{eq: equivalence 1}. 
    For all $x,y\in \A$ and $f\in \A^*$
    \begin{align*}
         (x \diamond y) \diamond f= (x\diamond f) \diamond y &\;\Longleftrightarrow\;    \rmR^*_f (x\cdot y) + \rmL^*_{x\cdot y}(f)= \rmR^*_f(x) \cdot y + \rmL^*_{L^*_x (f)}(y) +\rmR^*_y(\rmL^*_x(f)) \\
         &\;\Longleftrightarrow\; \begin{cases}
             \rmR^*_f (x\cdot y) =\rmR^*_f(x) \cdot y + \rmL^*_{L^*_x (f)}(y) \in \A \\
             \rmL^*_{x\cdot y}(f)=\rmR^*_y(\rmL^*_x(f))\in \A^*
         \end{cases}.
    \end{align*}
    Now, for all $z\in \A$
    \begin{align*}
        &\langle\rmL^*_{x\cdot y}(f), z\rangle=\langle\rmR^*_y(\rmL^*_x(f)), z \rangle \Longleftrightarrow \langle\rmR_{x\cdot y}(z), f\rangle=\langle \rmL_y(z), \rmL^*_x(f)\rangle
        \Longleftrightarrow \langle\rmR_{x\cdot y}(z), f\rangle=\langle \rmR_x(\rmL_y(z)),  f \rangle
        \\
        &\Longleftrightarrow \langle z\cdot (x\cdot y), f\rangle=\langle (y\cdot z)\cdot x,f\rangle  \Longleftrightarrow z\cdot (x\cdot y)=(y\cdot z)\cdot x,
    \end{align*}
    i.e., $(\A, \cdot)$ is a nearly associative algebra. 
 Moreover, for all $h\in \A^*$
 \begin{equation}\label{eq: prova}
 \langle\rmR^*_f (x\cdot y),h \rangle=\langle\rmR^*_f(x) \cdot y,h \rangle + \langle\rmL^*_{L^*_x (f)}(y), h \rangle.
\end{equation}
The left-hand side of \eqref{eq: prova}  is $\langle\rmR^*_f (x\cdot y),h \rangle=
\langle x\cdot y, f \star h\rangle=
\langle\Delta_\A(x\cdot y), f \otimes h\rangle$. On the right-hand, we get
\begin{align*}
        \langle\rmR^*_f(x) \cdot y,h \rangle + \langle\rmL^*_{L^*_x (f)}(y), h \rangle &= \langle \rmR^*_f(x) \cdot y, h\rangle + \langle \rmL^*_{f\rmR_x}(y), h\rangle= \langle \rmR^*_f(x) \cdot y, h\rangle + \langle y, h\star f\rmR_x\rangle\\
        &= \langle \Delta_{\A^*} (h), \rmR^*_f(x) \otimes y\rangle + \langle \Delta_\A(y), h \otimes f\rmR_x\rangle\\
        &= \langle \sum_{(h)} h_{(1)} \otimes  h_{(2)}, \rmR^*_f(x) \otimes y\rangle + \langle \sum_{(y)} y_{(1)} \otimes  y_{(2)}, h \otimes f\rmR_x\rangle\\
        &=\sum_{(h)} \langle x, f \star h_{(1)}\rangle \,  \langle y, h_{(2)}\rangle+ \sum_{(y)} \langle y_{(1)}, h\rangle \, \langle y_{(2)}\cdot x, f\rangle\\
        &=\sum_{(h)} \langle \Delta_\A(x), f \otimes h_{(1)}\rangle \, \langle y, h_{(2)}\rangle+ \sum_{(y)} \langle \tau(y_{(1)} \otimes  y_{(2)}\cdot x), f\otimes h\rangle\\
           &= \sum_{(x)} \langle x_{(1)}, f \rangle \, \langle x_{(2)} \otimes y ,\sum_{(h)}h_{(1)}\otimes h_{(2)}\rangle+ \sum_{(y)} \langle \tau(y_{(1)} \otimes  y_{(2)}\cdot x), f\otimes h\rangle\\
            &= \sum_{(x)} \langle x_{(1)}, f \rangle \, \langle x_{(2)} \otimes y ,\Delta_{\A^*} (h)\rangle+ \sum_{(y)} \langle \tau(y_{(1)} \otimes  y_{(2)}\cdot x), f\otimes h\rangle\\
            &=\sum_{(x)} \langle x_{(1)}, f \rangle \, \langle x_{(2)} \cdot y, h\rangle+ \sum_{(y)} \langle \tau(y_{(1)} \otimes y_{(2)}\cdot x), f\otimes h\rangle\\
             &=\langle\sum_{(x)} x_{(1)} \otimes x_{(2)} \cdot y, f \otimes h\rangle+ \langle\sum_{(y)} \tau(y_{(1)} \otimes y_{(2)}\cdot x), f\otimes h\rangle.
    \end{align*}
   Therefore, Equation  \eqref{eq: prova} is equivalent to 
  \begin{align*}
               & \Delta_\A(x\cdot y)=\sum_{(x)}  x_{(1)} \otimes x_{(2)} \cdot y+ \tau\big(\sum_{(y)} y_{(1)} \otimes y_{(2)}\cdot x\big)\Longleftrightarrow \Delta_\A (x \cdot y) = (\mbox{I} \otimes \rmR_y) \Delta_\A(x)+ \tau(\mbox{I} \otimes \rmR_x)\Delta_\A (y).
    \end{align*} 
This proves the Equivalence \eqref{eq: equivalence 1}, as required.

Now,  let us show  Equivalence \eqref{eq: equivalence 2}.  For all $x,y\in \A$ and $f\in \A^*$, we have 
\begin{equation*} 
    (f \diamond x) \diamond y =
          \rmL^*_f (x)  \diamond y +R^*_x (f) \diamond y =
          \rmL^*_f(x) \cdot y  + \rmL^*_{R^*_x (f)}(y)  
         +  \rmR^*_y(\rmR^*_x(f))  ,
\end{equation*}
and similar for $(f \diamond y) \diamond x $ just  by exchanging $x$ with $y$. Thus,
 \begin{align*}
          (f \diamond x) \diamond y= (f \diamond y) \diamond x\;\Longleftrightarrow\; \begin{cases}
             \rmL^*_f(x) \cdot y  + \rmL^*_{R^*_x (f)}(y) =\rmL^*_f(y) \cdot x  + \rmL^*_{R^*_y (f)}(x)  \in \A\\
        \rmR^*_y(\rmR^*_x(f))=\rmR^*_x(\rmR^*_y(f))\in \A^*
         \end{cases}
    \end{align*}
Now, for all $z\in \A$
    \begin{align*}
         \langle \rmR^*_y(\rmR^*_x(f)), z\rangle=  \langle \rmL_y(z), \rmR^*_x(f)\rangle
       =\langle \rmL_x(\rmL_y(z)),  f \rangle =\langle x \cdot( y \cdot z)  ,f\rangle,
    \end{align*}
   and similarly $\langle\rmR^*_x(\rmR^*_y(f)), z \rangle=\langle y \cdot( x \cdot z),f \rangle$. Thus, 
   \begin{equation*}
        \rmR^*_y(\rmR^*_x(f))=\rmR^*_x(\rmR^*_y(f)) \;\Longleftrightarrow\; x \cdot( y \cdot z)=y \cdot( x \cdot z)
   \end{equation*}
   for all $x,y,z\in \A$ and $f\in \A^*$, i.e., $(\A, \cdot)$ is an $L$-algebra.
 Moreover, for all $h\in \A^*$  we get
 
\begin{equation*}
\begin{split}
        \langle  \rmL^*_f(x) \cdot y  + \rmL^*_{R^*_x (f)}(y), h \rangle &= \langle \rmL^*_f(x) \cdot y, h\rangle + \langle \rmL^*_{f\rmL_x}(y), h\rangle= \langle \rmL^*_f(x) \cdot y, h\rangle + \langle y, h\star f\rmL_x\rangle\\
        &= \langle \Delta_{\A^*} (h), \rmL^*_f(x) \otimes y\rangle + \langle \Delta_\A(y), h \otimes f\rmL_x\rangle\\
        &= \langle \sum_{(h)} h_{(1)} \otimes  h_{(2)}, \rmL^*_f(x) \otimes y\rangle + \langle \sum_{(y)} y_{(1)} \otimes  y_{(2)}, h \otimes f\rmL_x\rangle\\
        & =\sum_{(h)} \langle x,   h_{(1)} \star f\rangle \, \langle y, h_{(2)}\rangle+ \sum_{(y)} \langle y_{(1)}, h\rangle \, \langle x \cdot y_{(2)}  , f\rangle\\
        &  =\sum_{(h)} \langle \Delta_\A(x),    h_{(1)}\otimes f\rangle \, \langle y, h_{(2)}\rangle+ \sum_{(y)} \langle y_{(1)} \otimes x \cdot y_{(2)}  , h\otimes f\rangle\\
&   =\sum_{(h)} \sum_{(x)} \langle  x_{(1)},h_{(1)}\rangle \, \langle x_{(2)},      f\rangle \, \langle y, h_{(2)}\rangle+ \sum_{(y)} \langle y_{(1)} \otimes x \cdot y_{(2)}  , h\otimes f\rangle\\
           &    = \sum_{(x)} \langle x_{(2)}, f \rangle \, \langle x_{(1)} \otimes y ,\sum_{(h)}h_{(1)}\otimes h_{(2)}\rangle+\sum_{(y)} \langle y_{(1)} \otimes x \cdot y_{(2)}  , h\otimes f\rangle\\
            &  =\sum_{(x)} \langle x_{(2)}, f \rangle \, \langle x_{(1)} \otimes y ,\Delta_{\A^*} (h)\rangle+\sum_{(y)} \langle y_{(1)} \otimes x \cdot y_{(2)}  , h\otimes f\rangle\\
            &  =\sum_{(x)} \langle x_{(2)}, f \rangle \, \langle x_{(1)} \cdot y, h\rangle+ \sum_{(y)} \langle y_{(1)} \otimes x \cdot y_{(2)}  , h\otimes f\rangle\\
             &=\langle\sum_{(x)}  x_{(1)}\cdot y \otimes x_{(2)}, h \otimes f\rangle+ \sum_{(y)} \langle y_{(1)} \otimes x \cdot y_{(2)}  , h\otimes f\rangle.
\end{split}
    \end{equation*}
  Doing similar calculations, just  by exchanging $x$ with $y$, we obtain that 
  \begin{equation*}
       \langle  \rmL^*_f(y) \cdot x  + \rmL^*_{R^*_y (f)}(x), h \rangle =\langle\sum_{(y)}  y_{(1)}\cdot x \otimes y_{(2)}, h \otimes f\rangle+ \sum_{(x)} \langle x_{(1)} \otimes y \cdot x_{(2)}  , h\otimes f\rangle.
  \end{equation*}
  Thus, we conclude that, for all $x,y\in \A$, $f\in \A^*$,
  \begin{align*}
               \rmL^*_f(x) \cdot y  + \rmL^*_{R^*_x (f)}(y) =\rmL^*_f(y) \cdot x  + \rmL^*_{R^*_y (f)}(x) \;\Longleftrightarrow\; \Delta_\A(x)+ (\mbox{I} \otimes \rmL_x)\Delta_\A (y)= (\mbox{I} \otimes \rmL_y)\Delta_\A (x) + ( \rmR_x  \otimes\mbox{I} ) \Delta_\A(y),  
    \end{align*} 
which proves Equivalence \eqref{eq: equivalence 2}, as desired.  The two remaining equivalences are proved similarly to those shown above, finishing the proof.
\end{proof}

We obtain a key characterization of nearly associative bialgebras as a consequence of Theorem \ref{teo: characterization nearly associative bialg} and Proposition \ref{pro:  nearly coasso coalg and L-coalg}.
\begin{teo}\label{teo: nearly ass bialgebras comultiplication}
     Let $(\A, \cdot)$ be a nearly associative $L$-algebra and $\Delta_{\A}: \A \longrightarrow \A \otimes \A$ be a comultiplication of $\A$.  Then $(\A, \cdot , \Delta_{\A}  )$  is a nearly associative bialgebra if and only if $(\A, \Delta_{\A})$ is a nearly coassociative $L$-coalgebra and 
     Conditions \eqref{eq: 1 nearly asso bialgebra} and  \eqref{eq: 2 nearly asso bialgebra} hold.
\end{teo}

\begin{exa}\label{ex: NA-bilagebra dim 4}
    Let $ \A$ be a vector space of dimension four with basis $\{e_1,e_2,e_3,e_4 \} $, and consider the bilinear product $\cdot:\A \times \A \longrightarrow \A$ defined by
    $$
    e_1\cdot e_1 = e_2 \cdot e_2=e_4, \quad e_1 \cdot e_2=e_3+e_4, \quad e_2 \cdot e_1=-e_3+e_4,
    $$
    and all other products equal to $0$. Then $(\A, \cdot)$ is nearly associative $L$-algebra isomorphic to the one defined in the Example \ref{ex: NAL-algebra associaitve}. Also, if we consider the comultiplication $\Delta_{\A}: \A \longrightarrow \A \otimes \A$ defined in Example \ref{ex: nearly coassociative L-coalgebra asso}, then $(\A, \cdot , \Delta_{\A}  )$ is  a nearly associative bialgebra. In fact, by Example \ref{ex: nearly coassociative L-coalgebra asso} $(\A, \Delta_{\A})$ is a nearly coassociative $L$-coalgebra and by straightforward calculations Conditions \eqref{eq: 1 nearly asso bialgebra} and  \eqref{eq: 2 nearly asso bialgebra} hold.
\end{exa}

\begin{exa}\label{ex: NA-bialgebra dim 6}
    Let $(\A, \cdot)$ be the six dimensional nearly associative $L$-algebra defined in Example \ref{ex: NAL-algebra} and $\Delta_{\A}: \A \longrightarrow \A \otimes \A$ the comultiplication defined in Example \ref{ex: nearly coassociative L-coalgebra not asso}. Then $(\A, \cdot , \Delta_{\A}  )$ is  a nearly associative bialgebra. In fact, by Example \ref{ex: nearly coassociative L-coalgebra not asso} $(\A, \Delta_{\A})$ is a nearly coassociative $L$-coalgebra and by straightforward calculations Conditions \eqref{eq: 1 nearly asso bialgebra} and  \eqref{eq: 2 nearly asso bialgebra} hold.
\end{exa}

\section{Coboundary nearly associative bialgebras,  $LR$-Yang-Baxter equation, and $r-$matrices} \label{Section: Coboundary NA bialgebras}

In the previous section, we studied nearly associative bialgebras in general, now we focus on a special class of them, namely coboundary nearly associative bialgebras. This class gives us a natural framework to investigate the analog of the Yang-Baxter equation and the matrix $r$ in our context.

Let  $(A, \cdot)$ be an algebra and consider  $  r = \sum_{i=1}^n a_i \otimes b_i \in A \otimes A$. Since in general in our structure we do not have the unit, we write directly the following elements in $A\otimes A \otimes A$ (as in \cite{Zhe1997}):

 \begin{align*}
 &r_{12} r_{23}  = \sum_{i,j=1}^n a_i \otimes  (b_i \cdot a_j ) \otimes b_j,\\
& r_{13} r_{12}   = \sum_{i,j=1}^n
 (a_i\cdot a_j)
\otimes b_j \otimes b_i,\\
&r_{23} r_{13} = \sum_{i,j=1}^n a_j \otimes
a_i\otimes (b_i \cdot b_j),\\
& \mathbf{LR}(r):= r_{12}r_{23}-r_{13}r_{12}-r_{23}r_{13}.
\end{align*}
Recall that $\displaystyle r \in A\otimes A$ is \textit{skew-symmetric} if $\displaystyle
r +\tau(r)=0$, where $\tau: A \otimes A \longrightarrow A \otimes A$ is the twisted map. 
Now, we define the comultiplication $\displaystyle   \Delta_{r}: A  \longrightarrow A \otimes A $ by,  for  $x \in A$,
\begin{equation}\label{eq:comultiplication r}
 \begin{split}
\Delta_{r} (x) & : =( \rmL_x \otimes \mbox{ I}- \mbox{I} \otimes \rmR_x)(r)=\sum_{i=1}^n \Big((x\cdot a_i) \otimes b_i  -a_i \otimes ( b_i \cdot x) \Big).
 \end{split}
\end{equation}
Also, for $x \in  A$, we consider the operator $\ad_x:=\rmL_x - \rmR_x$ of $A$, i.e., $\ad_x(y)=x\cdot y -y \cdot x$ for all $y\in A$. Moreover, as in Section \ref{Section: Coalgebras}, we denote the product ``$\cdot$" of $A$ also by $m: A \otimes A \longrightarrow A$, that is,  $m(x\otimes y):= x\cdot y,\,$  for   $ x,y \in A$, i.e., $(A,m=\cdot)$.

\begin{pro}\label{pro: nearly coass coalg}
    Let $(A, m=\cdot)$ be an algebra and $r \in A \otimes A$ be skew-symmetric. If $\Delta_r$ is the comultiplication defined by \eqref{eq:comultiplication r}, then $(A, \Delta_r)$ is a nearly coassociative coalgebra if and only if, for all $x\in A$,
    \begin{equation*}
        \Big( \big(m (\ad_x \otimes \mbox{I}) \otimes \mbox{I} \otimes \mbox{I} + m \otimes \mbox{I} \otimes \ad_x \big)(\mbox{I} \otimes \xi)- \big(\mbox{I} \otimes m (\mbox{I} \otimes \ad_x) \otimes \mbox{I} + \mbox{I} \otimes m \otimes \ad_x \big)\Big)(r \otimes r)=0.
    \end{equation*}
\end{pro}
\begin{proof}
     Let $x\in A$, by \eqref{eq:comultiplication r} we have
    \begin{align*}
        (\Delta_r\otimes \mbox{I})\Delta_r(x)&=(\Delta_r\otimes \mbox{I}) \Big(\sum_{i=1}^n \Big((x\cdot a_i) \otimes b_i  -a_i \otimes ( b_i \cdot x) \Big)\Big)= \sum_{i=1}^n \Big(\Delta_r (x\cdot a_i) \otimes b_i  - \Delta_r (a_i) \otimes ( b_i \cdot x) \Big)\\
        &= \sum_{i,j=1}^n \Big(((x\cdot a_i)\cdot a_j)\otimes b_j \otimes b_i 
        - a_j \otimes (b_j \cdot (x\cdot a_i))\otimes b_i
        -(a_i \cdot a_j) \otimes b_j \otimes ( b_i \cdot x)
        +  a_j \otimes (b_j \cdot a_i) \otimes ( b_i \cdot x)\Big)
    \end{align*}   
and
\begin{align*}
         \xi(\mbox{I} \otimes \Delta_r)\Delta_r(x)&=\xi (\mbox{I} \otimes \Delta_r) \Big(\sum_{i=1}^n \Big((x\cdot a_i) \otimes b_i  -a_i \otimes ( b_i \cdot x) \Big)\Big)= \xi \Big(\sum_{i=1}^n \Big((x\cdot a_i) \otimes \Delta_r(b_i)  -a_i \otimes \Delta_r( b_i \cdot x) \Big)\Big)\\
        &= \xi\Big(\sum_{i,j=1}^n \Big((x\cdot a_i) \otimes 
        (b_i\cdot a_j) \otimes b_j  -(x\cdot a_i) \otimes a_j \otimes ( b_j \cdot b_i)
        -a_i \otimes 
        (( b_i \cdot x)\cdot a_j) \otimes b_j  
        +a_i \otimes a_j \otimes ( b_j \cdot ( b_i \cdot x))\Big)\Big)\\
        &  =\sum_{i,j=1}^n \Big(
        (b_i\cdot a_j) \otimes b_j \otimes (x\cdot a_i) 
        -  a_j \otimes ( b_j \cdot b_i) \otimes (x\cdot a_i)
        -  (( b_i \cdot x)\cdot a_j) \otimes b_j  \otimes a_i
        +  a_j \otimes ( b_j \cdot ( b_i \cdot x))\otimes a_i\Big).
    \end{align*}
Then by Definition \ref{def: nearly coassociative coalgebra} $(A, \Delta_r)$ is a nearly coassociative coalgebra  (i.e., $(\Delta_r \otimes \mbox{I}) \Delta_r  =\xi (\mbox{I} \otimes \Delta_r) \Delta_r $)   if and only if
\begin{align}\label{eq: computation r NCC}
        &\sum_{i,j=1}^n \Big(
        ((x\cdot a_i)\cdot a_j)\otimes b_j \otimes b_i 
         +  (( b_i \cdot x)\cdot a_j) \otimes b_j  \otimes a_i\Big) 
         +\sum_{i,j=1}^n \Big(  -(a_i \cdot a_j) \otimes b_j \otimes ( b_i \cdot x)
        -(b_i\cdot a_j) \otimes b_j \otimes (x\cdot a_i) \Big)\nonumber \\
        & +\sum_{i,j=1}^n \Big( - a_j \otimes (b_j \cdot (x\cdot a_i))\otimes b_i
         -  a_j \otimes ( b_j \cdot ( b_i \cdot x))\otimes a_i \Big)
         +\sum_{i,j=1}^n \Big(   a_j \otimes (b_j \cdot a_i) \otimes ( b_i \cdot x)
        +  a_j \otimes ( b_j \cdot b_i) \otimes (x\cdot a_i)\Big)=0.
\end{align}
     Now, since $\tau(r)=-r$, we get that  Condition \eqref{eq: computation r NCC} is equivalent to
\begin{align*}
        &\sum_{i,j=1}^n \Big(
       ((x\cdot a_i)\cdot a_j)\otimes b_j \otimes b_i 
         -  (( a_i \cdot x)\cdot a_j) \otimes b_j  \otimes b_i\Big) 
         + \sum_{i,j=1}^n \Big( (a_i\cdot a_j) \otimes b_j \otimes (x\cdot b_i) 
         -(a_i \cdot a_j) \otimes b_j \otimes ( b_i \cdot x) \Big)\\
         &  +\sum_{i,j=1}^n \Big( 
          a_j \otimes ( b_j \cdot ( a_i \cdot x))\otimes b_i 
          - a_j \otimes (b_j \cdot (x\cdot a_i))\otimes b_i\Big)
       +\sum_{i,j=1}^n \Big( 
          a_j \otimes (b_j \cdot a_i) \otimes ( b_i \cdot x)
        -  a_j \otimes ( b_j \cdot a_i) \otimes (x\cdot b_i)\Big)=0.
\end{align*}
Thus, $(A, \Delta_r)$ is a nearly coassociative coalgebra if and only if 
\begin{align*}
       & \Big( \big(m (\ad_x \otimes \mbox{I}) \otimes \mbox{I} \otimes \mbox{I} + m \otimes \mbox{I} \otimes \ad_x \big)(\mbox{I} \otimes \xi)- \big(\mbox{I} \otimes m (\mbox{I} \otimes \ad_x) \otimes \mbox{I} + \mbox{I} \otimes m \otimes \ad_x \big)\Big)(r \otimes r)=0,
\end{align*}
as required.
\end{proof}

    \begin{exa}\label{ex: Nearly coalgebra r_2,6}
      Let $(\A, \cdot)$ be the six dimensional nearly associative $L$-algebra defined in Example \ref{ex: NAL-algebra}, and consider $r_{2,6}=e_2 \otimes e_6 - e_6 \otimes e_2\in \A \otimes \A$. If $\Delta_{r_{2,6}}$ is the comultiplication defined by \eqref{eq:comultiplication r}, then, for $e_i\in \A$, $1\leq i \leq 6$,
 \begin{align*} 
\Big( (m  (\ad_{e_i} & \otimes \mbox{I}) \otimes \mbox{I} \otimes \mbox{I} + m \otimes \mbox{I} \otimes \ad_{e_i})(\mbox{I} \otimes \xi)- (\mbox{I} \otimes m (\mbox{I} \otimes \ad_{e_i}) \otimes \mbox{I} + \mbox{I} \otimes m \otimes \ad_{e_i})\Big)(r \otimes r)= \nonumber\\
        =&
       ((e_i\cdot e_2)\cdot e_2)\otimes e_6 \otimes e_6 
         -  (( e_2 \cdot e_i)\cdot e_2) \otimes e_6  \otimes e_6 
          -e_6 \otimes ( e_2 \cdot ( e_2 \cdot e_i))\otimes e_6 
          + e_6 \otimes (e_2 \cdot (e_i\cdot e_2))\otimes e_6.
\end{align*}
Since for $i\neq 1$,
$(e_i\cdot e_2)\cdot e_2 = ( e_2 \cdot e_i)\cdot e_2 = e_2 \cdot ( e_2 \cdot e_i) = e_2 \cdot (e_i\cdot e_2) =0$, and for $i=1$,
\begin{align*}
    ((e_1\cdot e_2) \cdot e_2)\otimes e_6  \otimes e_6 
         -  (( e_2 \cdot e_1)\cdot e_2) \otimes e_6  \otimes e_6 
          -e_6 \otimes ( e_2 \cdot ( e_2 \cdot e_1))\otimes e_6 
          + e_6 \otimes (e_2 \cdot (e_1\cdot e_2))\otimes e_6 &=\\
          = (e_3 \cdot e_2)\otimes e_6 \otimes e_6 
         -  (e_4\cdot e_2) \otimes e_6  \otimes e_6 
          -e_6 \otimes ( e_2 \cdot e_4)\otimes e_6 
          + e_6 \otimes (e_2 \cdot e_3)\otimes e_6
          &=0,
\end{align*}
 it follows that $(A, \Delta_{r_{2,6}})$ is a nearly coassociative coalgebra.
\end{exa}

\begin{pro}\label{pro: L-coalg, R-coalg r-matrices}
    Let $(\A, \cdot)$ be a nearly associative $L$-algebra  and $r \in \A \otimes \A$ be skew-symmetric. If $\Delta_r$ is the comultiplication defined by \eqref{eq:comultiplication r}, then we have the following assertions:
    \begin{enumerate}
        \item $(\A, \Delta_r)$ is an $L$-coalgebra if and only if $\big((\rmL_x \otimes \mbox{I} \otimes \mbox{I})- (\tau \otimes \mbox{I})(\rmL_x \otimes \mbox{I} \otimes \mbox{I})\big)\mathbf{LR}(r)=0.$
        
        \item $(\A, \Delta_r)$ is an $R$-coalgebra if and only if
        $\big((\mbox{I}\otimes \mbox{I}\otimes\rmR_x) -  (\mbox{I}\otimes \tau)(\mbox{I}\otimes \mbox{I}\otimes\rmR_x)\big)\mathbf{LR}(r)=0.$
    \end{enumerate}
\end{pro}
\begin{proof}
Let us prove Equivalence (1). The Equivalence (2) is proven similarly.
Let $x\in \A$. By \eqref{eq:comultiplication r} we have
\begin{align*}
        (\mbox{I} \otimes \Delta_r)\Delta_r(x)&=(\mbox{I} \otimes \Delta_r) \Big(\sum_{i=1}^n \Big((x\cdot a_i) \otimes b_i  -a_i \otimes ( b_i \cdot x) \Big)\Big)= \sum_{i=1}^n \Big((x\cdot a_i) \otimes \Delta_r(b_i)  -a_i \otimes \Delta_r( b_i \cdot x) \Big)\\
        &= \sum_{i,j=1}^n \Big((x\cdot a_i) \otimes 
        (b_i\cdot a_j) \otimes b_j  -(x\cdot a_i) \otimes a_j \otimes ( b_j \cdot b_i)
        -a_i \otimes 
        (( b_i \cdot x)\cdot a_j) \otimes b_j  
        +a_i \otimes a_j \otimes ( b_j \cdot ( b_i \cdot x))\Big).
    \end{align*}
and, consequently,
     \begin{align*}
        (\tau \otimes \mbox{I})(\mbox{I} \otimes \Delta_r)\Delta_r(x)
        &= \sum_{i,j=1}^n \Big( 
        (b_i\cdot a_j) \otimes (x\cdot a_i) \otimes b_j 
        - a_j \otimes (x\cdot a_i) \otimes( b_j \cdot b_i)
        -
        (( b_i \cdot x)\cdot a_j) \otimes a_i \otimes  b_j  
        + a_j \otimes a_i \otimes ( b_j \cdot ( b_i \cdot x))\Big).
    \end{align*}    
 Then by Definition \ref{def: LR coalgebras} $(\A, \Delta_r)$ is an $L$-coalgebra (that is, $ (\mbox{I}\otimes \Delta_r) \Delta_r  =(\tau \otimes \mbox{I})( \mbox{I} \otimes \Delta_r)\Delta_r$) if and only if 
     \begin{align}\label{eq: computation r L-alg}
        &\sum_{i,j=1}^n \Big((x\cdot a_i) \otimes 
        (b_i\cdot a_j) \otimes b_j  -(x\cdot a_i) \otimes a_j \otimes ( b_j \cdot b_i)
        +
        (( b_i \cdot x)\cdot a_j) \otimes a_i \otimes  b_j  
        \Big) \nonumber\\
        &+\sum_{i,j=1}^n \Big(a_i \otimes a_j \otimes ( b_j \cdot ( b_i \cdot x))
         - a_j \otimes a_i \otimes ( b_j \cdot ( b_i \cdot x))\Big) \nonumber\\
         &+ \sum_{i,j=1}^n \Big( 
          a_j \otimes (x\cdot a_i) \otimes( b_j \cdot b_i)
        -(b_i\cdot a_j) \otimes (x\cdot a_i) \otimes b_j 
        -a_i \otimes 
        (( b_i \cdot x)\cdot a_j) \otimes b_j  
       \Big)=0.
    \end{align}
       Since $\A$ is an $NAL$-algebra  and $\tau(r)=-r$, it follows that Equation \eqref{eq: computation r L-alg} is equivalent to
       \begin{align*}
        &\sum_{i,j=1}^n \Big((x\cdot a_i) \otimes 
        (b_i\cdot a_j) \otimes b_j  -(x\cdot a_i) \otimes a_j \otimes ( b_j \cdot b_i)
        -
        (x\cdot ( a_j \cdot a_i)) \otimes b_i \otimes  b_j  
        \Big) \nonumber\\
         &+ \sum_{i,j=1}^n \Big( 
          a_j \otimes (x\cdot a_i) \otimes( b_j \cdot b_i)
        +(a_i\cdot a_j) \otimes (x\cdot b_i) \otimes b_j 
        -a_i \otimes 
        (x \cdot (a_j \cdot b_i)) \otimes b_j  
       \Big)=0.
    \end{align*}
So, as a consequence, since $r$ is skew-symmetric, Equation \eqref{eq: computation r L-alg} is also equivalent to
    \begin{align*}
         &\sum_{i,j=1}^n \Big((x\cdot a_i) \otimes 
        (b_i\cdot a_j) \otimes b_j  
        -(x\cdot a_i) \otimes a_j \otimes ( b_j \cdot b_i)
        -(x\cdot ( a_j \cdot a_i)) \otimes b_i \otimes  b_j  \Big)\\
        &+ (\tau \otimes \mbox{I})\Big(\sum_{i,j=1}^n \Big(  -(x\cdot a_i) \otimes (b_i\cdot a_j) \otimes b_j 
        +  (x\cdot a_i) \otimes a_j \otimes( b_j \cdot b_i) 
        + (x \cdot ( a_j \cdot a_i)) \otimes b_i \otimes b_j \Big)
        \Big)=0.
    \end{align*}
Thus, $(\A, \Delta_r)$ is an $L$-coalgebra if and only if 
     \begin{align*}
        \big(\rmL_x \otimes \mbox{I} \otimes \mbox{I}- (\tau \otimes \mbox{I})(\rmL_x \otimes \mbox{I} \otimes \mbox{I})\big)(r_{12}r_{23}-r_{23}r_{13}-r_{13}r_{12})=0,
    \end{align*}
    as required.
\end{proof}

   \begin{exa}\label{ex: L-coalgebra r_2,6}  
      Let $(\A, \cdot)$ be the six dimensional nearly associative $L$-algebra defined in Example \ref{ex: NAL-algebra}, and consider $r_{2,6}=e_2 \otimes e_6 - e_6 \otimes e_2\in \A \otimes \A$. If $\Delta_{r_{2,6}}$ is the comultiplication defined by \eqref{eq:comultiplication r}, then, for $e_i\in \A$, $1\leq i \leq 6$,
 \begin{align*}
& \big((\rmL_{e_i} \otimes \mbox{I} \otimes \mbox{I})- (\tau \otimes \mbox{I})(\rmL_{e_i} \otimes \mbox{I} \otimes \mbox{I})\big) \mathbf{LR}(r)=  
- (e_i \cdot e_5 )\otimes e_6 \otimes e_6  +  e_6 \otimes (e_i \cdot e_5 )
\otimes  e_6 =0.
\end{align*}
      Consequently, $(A, \Delta_{r_{2,6}})$ is an   $L$-coalgebra.
\end{exa}

\begin{defi}
A nearly associative bialgebra $(\A, \cdot, \Delta_\A)$ is called \textit{coboundary} if there exists $\displaystyle r \in \A \otimes \A$ such that  $\Delta_\A (x) =\Delta_{r} (x)$, for all $x\in \A$, where $ \Delta_{r} $ is the comultiplication defined by 
 \eqref{eq:comultiplication r}.

\end{defi}

\begin{exa}
     Let $(\A, \cdot, \Delta_\A)$ be the four dimensional nearly associative bialgebra of Example \ref{ex: NA-bilagebra dim 4}. If we consider $r=e_1 \otimes e_4 - e_4 \otimes e_1\in \A \otimes \A$, then $\Delta_\A (x) =\Delta_{r} (x)= x\cdot e_1 \otimes e_4 + e_4 \otimes e_1 \cdot x$, for all $x \in \A$. Consequently, $(\A, \cdot, \Delta_\A)$ is a coboundary nearly associative bialgebra.
\end{exa}

\begin{exa}
    Let $(\A, \cdot, \Delta_\A)$ be the six dimensional nearly associative bialgebra of Example \ref{ex: NA-bialgebra dim 6}. If we consider $r_{5,6}=e_5 \otimes e_6 - e_6 \otimes e_5\in \A \otimes \A$, then $\Delta_\A (x) =\Delta_{r_{5,6}} (x)$, for all $x \in \A$. Consequently, $(\A, \cdot, \Delta_\A)$ is a coboundary nearly associative bialgebra.
\end{exa}

\begin{rmk}
    Let $(\A, \cdot)$ be a nearly associative $L$-algebra  and $r \in \A \otimes \A$ be skew-symmetric. If $\Delta_r$ is the comultiplication defined by \eqref{eq:comultiplication r}, then $(\A, \cdot, \Delta_r)$ is not always a nearly associative bialgebra.
\end{rmk}

\begin{exa}
     Let $(\A, \cdot)$ be the six dimensional nearly associative $L$-algebra defined in Example \ref{ex: NAL-algebra} and $r_{1,5}= e_1 \otimes e_5 - e_5 \otimes e_1$. Consider the comultiplication $\Delta_{r_{1,5}}: \A \longrightarrow \A \otimes \A$ defined by \eqref{eq:comultiplication r}, i.e., for all $x\in \A$
     $$
     \Delta_{r_{1,5}}(x)=x\cdot e_1 \otimes e_5 - e_1 \otimes e_5 \cdot x - x \cdot e_5 \otimes e_1 + e_5 \otimes e_1 \cdot x.
     $$
     Since
         $(\Delta_{r_{1,5}} \otimes \mbox{I}) \Delta_{r_{1,5}}(e_2)
         = e_6 \otimes e_5 \otimes e_3+ e_5 \otimes e_6 \otimes e_3$ and 
         $(\mbox{I} \otimes \Delta_{r_{1,5}}) \Delta_{r_{1,5}}(e_2)
         = e_4 \otimes e_6 \otimes e_5 + e_4 \otimes e_5 \otimes e_6$,
     it follows that
     \begin{align*}
(\Delta_{r_{1,5}}\otimes \mbox{I}) \Delta_{r_{1,5}}(e_2)&=e_6 \otimes e_5 \otimes e_3+ e_5 \otimes e_6 \otimes e_3 \neq 
 e_6 \otimes e_5 \otimes e_4 +  e_5 \otimes e_6 \otimes e_4
=  \xi (\mbox{I} \otimes \Delta_{r_{1,5}}) \Delta_{r_{1,5}}(e_2),\\
 (\mbox{I}\otimes \Delta_{r_{1,5}}) \Delta_{r_{1,5}}(e_2) &= e_4 \otimes e_6 \otimes e_5 + e_4 \otimes e_5 \otimes e_6 \neq
   e_6 \otimes e_4 \otimes e_5 + e_5 \otimes e_4 \otimes  e_6
 =(\tau \otimes \mbox{I})( \mbox{I} \otimes \Delta_{r_{1,5}})\Delta_{r_{1,5}}(e_2),\\
 (\Delta_{r_{1,5}} \otimes\mbox{I}) \Delta_{r_{1,5}}(e_2) &=e_6 \otimes e_5 \otimes e_3+ e_5 \otimes e_6 \otimes e_3 \neq 
 e_6 \otimes e_3 \otimes e_5 + e_5 \otimes e_3 \otimes e_6 
 =(\mbox{I}\otimes \tau)( \Delta_{r_{1,5}} \otimes\mbox{I})\Delta_{r_{1,5}} (e_2).
\end{align*} 
     Then $(\A, \Delta_{r_{1,5}})$ is neither a nearly coassociative coalgebra nor an  $L$-coalgebra nor an  $R$-coalgebra. Consequently, $(\A, \cdot, \Delta_{r_{1,5}})$ is not a nearly associative bialgebra.  
\end{exa}


\begin{teo} \label{teo: coboundary NALR-algebra}
 Let $(\A, \cdot )$ be a nearly associative $L$-algebra and $r \in \A \otimes \A$ be skew-symmetric. If $\Delta_r$ is the comultiplication defined by \eqref{eq:comultiplication r}, then the triple $(\A, \cdot, \Delta_r)$ is a   coboundary  nearly associative bialgebra   if and only if  the following conditions hold:
\begin{enumerate}
\item $\big((\rmL_x \otimes \mbox{I} \otimes \mbox{I})- (\tau \otimes \mbox{I})(\rmL_x \otimes \mbox{I} \otimes \mbox{I})\big) \mathbf{LR}(r)=0$,

\vspace{1mm}
\item  $\Big( (m (\ad_x \otimes \mbox{I}) \otimes \mbox{I} \otimes \mbox{I} + m \otimes \mbox{I} \otimes \ad_x)(\mbox{I} \otimes \xi)- (\mbox{I} \otimes m (\mbox{I} \otimes \ad_x) \otimes \mbox{I} + \mbox{I} \otimes m \otimes \ad_x)\Big)(r \otimes r)=0$,

\vspace{1mm}
\item $(\rmL_x \otimes \rmR_y- \rmR_x \otimes \rmL_y)(r)=0$,

\vspace{1mm}
\item $
  (\ad_y \rmR_x  \otimes \mbox{I} - \mbox{I} \otimes  \rmR_x \ad_y  )(r)=0$,

  \vspace{1mm}
\item $(\rmL_y \otimes \rmL_x- \rmR_y \otimes \rmR_x)(r)=0$,

\vspace{1mm}
\item $(\mbox{I} \otimes \mbox{I} -\tau)(  \ad_{x\cdot y}  \otimes \mbox{I} )(r)=0$,
 \end{enumerate}
 for all $ x,y\in \A$.
\end{teo}

\begin{proof}
By Theorem \ref{teo: nearly ass bialgebras comultiplication}, it is enough to prove that the Conditions (1)-(6) are satisfied if and only if  $(\A, \Delta_r)$ is a nearly coassociative $L$-coalgebra, and the Conditions \eqref{eq: 1 nearly asso bialgebra} and \eqref{eq: 2 nearly asso bialgebra} hold. 

By Propositions \ref{pro: nearly coass coalg} and \ref{pro: L-coalg, R-coalg r-matrices}  $(\A, \Delta_r)$ is a nearly coassociative $L$-coalgebra if and only if Conditions (1) and (2) are sati\-sfied. Thus, we have just to prove that the Conditions (3)-(6) are satisfied if and only if the Conditions \eqref{eq: 1 nearly asso bialgebra} and \eqref{eq: 2 nearly asso bialgebra} hold.

Let $x,y\in \A$ and suppose that the Conditions \eqref{eq: 1 nearly asso bialgebra} and \eqref{eq: 2 nearly asso bialgebra} are satisfied. 
By definition of $\Delta_r$ and since $\tau(r)=-r$, we have that
\begin{align*}
    (\mbox{I} \otimes \rmR_y) \Delta_r(x)+ \tau(\mbox{I} \otimes \rmR_x)\Delta_r (y)
    &=\sum_{i=1}^n \Big((x\cdot a_i) \otimes( b_i \cdot y)  -a_i \otimes \big(( b_i \cdot x)\cdot y \big) 
    -  (a_i\cdot x) \otimes (y\cdot b_i) + \big(( a_i \cdot y)\cdot x \big) \otimes b_i\Big)
\end{align*}
and
     \begin{align*} 
         \tau(\rmL_y \otimes \mbox{ I}) \Delta_r (x) + (\rmL_x \otimes \mbox{ I}) \Delta_r (y)
     &= \sum_{i=1}^n \Big(-a_i \otimes \big(y\cdot (x\cdot b_i)\big)  + ( a_i \cdot x) \otimes (y \cdot b_i) 
    +\big(x \cdot (y\cdot a_i) \big)\otimes b_i  -(x \cdot a_i) \otimes ( b_i \cdot y) 
    \Big).
    \end{align*}
Since $\A$ is an $NAL$-algebra, we have
\begin{align*}
    (\mbox{I} \otimes \rmR_y) \Delta_r(x)+ \tau(\mbox{I} \otimes \rmR_x)\Delta_r (y)=\tau(\rmL_y \otimes \mbox{ I}) \Delta_r (x) + (\rmL_x \otimes \mbox{ I}) \Delta_r (y) \;
    &\Longleftrightarrow\;
    2\sum_{i=1}^n \big((x\cdot a_i) \otimes( b_i \cdot y)  
    -  (a_i\cdot x) \otimes (y\cdot b_i)\big)=0 \\ &\Longleftrightarrow\; 2(\rmL_x \otimes \rmR_y - \rmR_x \otimes \rmL_y)(r)=0. 
\end{align*}
Since the characteristic of $\mathbb{K}$ is different from $2$, Condition (3) is satisfied.

Recall that by Remark \ref{rmk: NAL-alg is R-alg} $\A$ is also $\rmR$-algebra. Now, since
    $\Delta_{r} (x \cdot y)
    = \sum_{i=1}^n \Big( \big(y\cdot (a_i \cdot x)\big) \otimes b_i  -a_i \otimes \big((y\cdot b_i)\cdot x\big) \Big)$,
by Condition (3) and using Proposition \ref{pro:LR cond NAL} it follows that
\begin{align*}
    &\Delta_{r} (x \cdot y)=(\mbox{I} \otimes \rmR_y) \Delta_r(x)+ \tau(\mbox{I} \otimes \rmR_x)\Delta_r (y)\\
    &\; \Longleftrightarrow\;  \sum_{i=1}^n \Big( \big(y\cdot (a_i \cdot x)\big) \otimes b_i  -a_i \otimes \big((y\cdot b_i)\cdot x \big)  +a_i \otimes \big(( b_i \cdot x)\cdot y \big) - \big(( a_i \cdot y)\cdot x \big) \otimes b_i\Big)=0\\
 &\; \Longleftrightarrow\;  \sum_{i=1}^n \Big( \big(\big(y\cdot (a_i \cdot x)\big) - \big(( a_i \cdot y)\cdot x \big)\big) \otimes b_i
 -a_i \otimes \big(\big((y\cdot b_i)\cdot x \big) - \big(( b_i \cdot x)\cdot y \big)\big) \Big)=0\\
  &\;\Longleftrightarrow\;  \big(( \rmL_y \rmR_x - \rmR_x \rmR_y)\otimes \mbox{I} - \mbox{I} \otimes ( \rmR_x\rmL_y - \rmR_x \rmR_y)\big)(r)=0\\
  &\;\Longleftrightarrow\;  \big(( \rmL_y \rmR_x - \rmR_y \rmR_x)\otimes \mbox{I} - \mbox{I} \otimes ( \rmR_x\rmL_y - \rmR_x \rmR_y)\big)(r)=0\\
    &\;\Longleftrightarrow\;  (\ad_y \rmR_x  \otimes \mbox{I} - \mbox{I} \otimes  \rmR_x \ad_y  )(r)=0.
\end{align*}
Thus, Condition (4) holds.

Now, by the definition of $\Delta_r$ and since $\tau(r)=-r$, it follows that
\begin{align*}
    (\rmR_y \otimes \mbox{ I}) \Delta_r(x) + (\mbox{I} \otimes \rmL_x) \Delta_r(y) 
    =\sum_{i=1}^n \Big(((x\cdot a_i)\cdot y) \otimes b_i  -(a_i \cdot y)\otimes ( b_i \cdot x) 
    + (y\cdot a_i) \otimes (x \cdot b_i)  -a_i \otimes ( x\cdot ( b_i \cdot y))
    \Big),
\end{align*}
\begin{align*} 
         \tau ( \mbox{I} \otimes \rmL_y) \Delta_r(x) +\tau (\rmR_x \otimes \mbox{ I}) \Delta_r (y)
     &=\sum_{i=1}^n \Big((y \cdot ( a_i \cdot x) )\otimes b_i - ( y \cdot a_i )\otimes (x\cdot b_i) + ( a_i \cdot y)\otimes (b_i \cdot x) - a_i \otimes ((y\cdot b_i) \cdot x )
     \Big),
    \end{align*}
and
\begin{align*}
     ( \mbox{I} \otimes \rmL_y) \Delta_r(x)+ (\rmR_x \otimes \mbox{ I}) \Delta_r (y) 
     &= \sum_{i=1}^n \Big((x\cdot a_i) \otimes( y \cdot b_i ) -a_i \otimes (y \cdot ( b_i \cdot x) )
     + ((y\cdot a_i) \cdot x )\otimes b_i  -(a_i \cdot x)\otimes ( b_i \cdot y)
     \Big).
\end{align*}
Since $\A$ is an $NAL$-algebra, it follows that
\begin{align*}
     (\rmR_y \otimes \mbox{ I}) \Delta_r(x) + (\mbox{I} \otimes \rmL_x) \Delta_r(y)=\tau ( \mbox{I} \otimes \rmL_y) \Delta_r(x) +\tau (\rmR_x \otimes \mbox{ I}) \Delta_r (y)\; &\Longleftrightarrow \; 2\sum_{i=1}^n \Big( (y\cdot a_i) \otimes (x \cdot b_i)  - ( a_i \cdot y)\otimes (b_i \cdot x) 
     \Big)=0 \\
     &\Longleftrightarrow \; 2(\rmL_y \otimes \rmL_x - \rmR_y \otimes \rmR_x)(r)=0.
    \end{align*}
Since the characteristic of $\mathbb{K}$ is different from $2$, Condition (5) is satisfied.

Now, by Condition (5) and using Proposition \ref{pro:LR cond NAL} we have
\begin{align*}
    & \tau ( \mbox{I} \otimes \rmL_y) \Delta_r(x) +\tau (\rmR_x \otimes \mbox{ I}) \Delta_r (y)= ( \mbox{I} \otimes \rmL_y) \Delta_r(x)+ (\rmR_x \otimes \mbox{ I}) \Delta_r (y)\\
    &\Longleftrightarrow \; \sum_{i=1}^n \Big( \big(y \cdot  ( a_i \cdot x)   
     -  (y\cdot a_i) \cdot x  \big)\otimes b_i
    + a_i \otimes \big( y \cdot ( b_i \cdot x)   - (y\cdot b_i) \cdot x  \big)
     \Big)=0\\
    &\Longleftrightarrow \; \big( (\rmL_y \rmR_x- \rmR_x \rmL_y) \otimes \mbox{I} + \mbox{I}\otimes ( \rmL_y \rmR_x- \rmR_x \rmL_y )\big)(r)=0\\
    &\Longleftrightarrow \; (\mbox{I} \otimes \mbox{I} -\tau)( (\rmL_y \rmR_x- \rmR_x \rmL_y) \otimes \mbox{I} )(r)=0\\
        &\Longleftrightarrow \; (\mbox{I} \otimes \mbox{I} -\tau)\big( (\rmL_{x \cdot y}  - \rmR_{x\cdot y}  ) \otimes \mbox{I} \big)(r)=0\\
     &\Longleftrightarrow \; (\mbox{I} \otimes \mbox{I} -\tau)(  \ad_{x\cdot y}  \otimes \mbox{I} )(r)=0.
\end{align*}
Thus, Condition (6) is satisfied. So, we proved that if the Conditions \eqref{eq: 1 nearly asso bialgebra} and \eqref{eq: 2 nearly asso bialgebra} are satisfied, then the Conditions (3)-(6) hold. The converse is proven with similar arguments, and this completes the proof.
\end{proof}

\begin{exa}\label{ex: NA-bialgebra dim 6 2,6}
     Let $(\A, \cdot)$ be the six dimensional nearly associative $L$-algebra defined in Example \ref{ex: NAL-algebra}, and consider $r_{2,6}=e_2 \otimes e_6 - e_6 \otimes e_2\in \A \otimes \A$. If $\Delta_{r_{2,6}}$ is the comultiplication defined by \eqref{eq:comultiplication r}, then by Examples \ref{ex: Nearly coalgebra r_2,6} and \ref{ex: L-coalgebra r_2,6} Conditions (1) and (2) of Theorem \ref{teo: coboundary NALR-algebra} are satisfied. Also, since $e_i \cdot e_6 = e_6 \cdot e_i =0$ for all $1\leq i \leq 6$, Conditions (3) and (5) of Theorem \ref{teo: coboundary NALR-algebra} hold. However,
\begin{align*}
 (\ad_{e_2} \rmR_{e_1}  \otimes \mbox{I} - \mbox{I} \otimes  \rmR_{e_1} \ad_{e_2}  )(r_{2,6})&=  
 \big(e_2\cdot (e_2 \cdot e_1)\big) \otimes e_6
  - \big(( e_2 \cdot e_1)\cdot e_2 \big) \otimes e_6
\\
 &=(e_2\cdot e_4) \otimes e_6
  - (e_4\cdot e_2 ) \otimes e_6
 =e_6 \otimes e_6\neq 0.
\end{align*}
Thus, Condition (4) of Theorem \ref{teo: coboundary NALR-algebra} does not hold, and as a consequence $(\A, \cdot, \Delta_{r_{2,6}})$ is not a nearly associative bialgebra.
\end{exa}

 \begin{defi} \label{def: LR(r)}
The equation 
$$\mathbf{LR}(r) = r_{12}r_{23}-r_{13}r_{12}-r_{23}r_{13}=0$$ 
is called \textit{$LR$-Yang-Baxter equation}  (or just {\it LRYBE } to abbreviate). If $ r \in A\otimes A$ satisfies $\mathbf{LR}(r)=0$, then $r$ is called the \textit{solution} of LRYBE, or \textit{$r$-matrix} of $A$. 
\end{defi}

\begin{exa}
     Let $(\A, \cdot, \Delta_\A)$ be the four dimensional coboundary nearly associative bialgebra of Example \ref{ex: NA-bilagebra dim 4}. Then $r=e_1 \otimes e_4 - e_4 \otimes e_1\in \A \otimes \A$ is a solution of LRYBE of $\A$. 
\end{exa}

\begin{exa}
    Let $(\A, \cdot, \Delta_\A)$ be the six dimensional coboundary nearly associative bialgebra of Example \ref{ex: NA-bialgebra dim 6}. Then $r_{5,6}=e_5 \otimes e_6 - e_6 \otimes e_5\in \A \otimes \A$  is a solution of LRYBE of $\A$. 
\end{exa}

\begin{exa}
     Let $(\A, \cdot)$ be the six dimensional nearly associative $L$-algebra defined in Example \ref{ex: NAL-algebra}, and consider $r_{2,6}=e_2 \otimes e_6 - e_6 \otimes e_2\in \A \otimes \A$. Then
     $$
     \mathbf{LR}(r_{2,6}) = - e_6 \otimes e_5 \otimes e_6 - e_5 \otimes e_6 \otimes e_6 - e_6 \otimes e_6 \otimes e_5\neq 0,
     $$
     and consequently $r_{2,6}$ is not a solution of LRYBE of $\A$. 
\end{exa}

We have the following immediate consequence of Theorem \ref{teo: coboundary NALR-algebra}.

\begin{cor}
Let $(\A, \cdot)$ be a nearly associative algebra.
If $r\in \A\otimes \A$ is a skew-symmetric solution of LRYBE  and $\Delta_r$ is the comultiplication defined by \eqref{eq:comultiplication r}, then $(\A, \Delta_r)$ is a $LR$-coalgebra.
\end{cor}

\begin{cor}
Let $(\A, \cdot)$ be a nearly associative $L$-algebra.
If $r\in \A\otimes \A$ is a skew-symmetric solution of LRYBE  and $\Delta_r$ is the comultiplication defined by \eqref{eq:comultiplication r}, then $(\A, \cdot, \Delta_r)$ is a coboundary nearly associative bialgebra if and only if  the following assertions are satisfied:
\begin{enumerate}
\item  $\Big( (m (\ad_x \otimes \mbox{I}) \otimes \mbox{I} \otimes \mbox{I} + m \otimes \mbox{I} \otimes \ad_x)(\mbox{I} \otimes \xi)- (\mbox{I} \otimes m (\mbox{I} \otimes \ad_x) \otimes \mbox{I} + \mbox{I} \otimes m \otimes \ad_x)\Big)(r \otimes r)=0$,

\vspace{1mm}
\item $(\rmL_x \otimes \rmR_y- \rmR_x \otimes \rmL_y)(r)=0$,

\vspace{1mm}
\item $
  (\ad_y \rmR_x  \otimes \mbox{I} - \mbox{I} \otimes  \rmR_x \ad_y  )(r)=0$,

  \vspace{1mm}
\item $(\rmL_y \otimes \rmL_x- \rmR_y \otimes \rmR_x)(r)=0$,

\vspace{1mm}
\item $(\mbox{I} \otimes \mbox{I} -\tau)(  \ad_{x\cdot y}  \otimes \mbox{I} )(r)=0$,
  \end{enumerate}
 for all $ x,y\in \A$.
\end{cor}

Now, we associate to $r = \sum_{i=1}^n a_i \otimes b_i \in A \otimes A$ the 
linear map $\displaystyle  \mathcal{R}: A^\ast \longrightarrow
A$ defined by, for  $f \in A^\ast$,
\begin{eqnarray} \displaystyle && \mathcal{R} (f)  := (f\otimes 1)(r) = \sum_{i=1}^n f(a_i)
b_i.
 \label{eq: R}
\end{eqnarray}

\begin{rmk}\label{rmk: r skew-symmetric R}
    $\displaystyle r$ is skew-symmetric if and only if  $\mathcal{R}(f)= -\sum_{i=1}^n f(b_i)a_i$ for all  $f \in A^\ast$. 
\end{rmk}

\begin{pro}
    Let $(\A, \cdot)$  be a nearly associative  $L$-algebra. Then $r\in \A\otimes \A$ is a skew-symmetric  solution of LRYBE if and only if 
 the linear map $\mathcal{R}$ defined by  \eqref{eq: R} verifies $\displaystyle \mathcal{R}(f)= -\sum_{i=1}^n f(b_i)a_i$ and
    \begin{equation}\label{eq: R LRYBE}
        \langle f, \mathcal{R}(h) \mathcal{R}(g) \rangle + \langle g, \mathcal{R}(f) \mathcal{R}(h) \rangle+ \langle h, \mathcal{R}(g) \mathcal{R}(f) \rangle=0,
    \end{equation}
for all  $f,g,h \in \A^\ast$.
\end{pro}
\begin{proof}
 First, assume that $ r = \sum_{i=1}^n a_i \otimes b_i \in \A \otimes \A$ is a skew-symmetric  solution of LRYBE, and 
let  $\mathcal{R}$ be the linear map defined by  \eqref{eq: R}. Since $\displaystyle r$ is skew-symmetric,
by Remark \ref{rmk: r skew-symmetric R} it remains to prove that Condition \eqref{eq: R LRYBE} holds.
To this end, notice that, as $r$ is a solution of LRYBE, then $\mathbf{LR}(r)=  
r_{12}r_{23}-r_{13}r_{12}-r_{23}r_{13}=0 $, which means,
\begin{align*}
 \sum_{i,j=1}^n \Big( a_i \otimes  (b_i \cdot a_j ) \otimes b_j - (a_i\cdot a_j)
\otimes b_j \otimes b_i -  a_j \otimes
a_i\otimes (b_i \cdot b_j) \Big)=0.
\end{align*}
As a consequence, for all $f,g,h\in \A^*$, we have 
 \begin{align*}
     &  \sum_{i,j=1}^n 
 \langle  f \otimes g \otimes h,  a_i \otimes  (b_i \cdot a_j ) \otimes b_j - (a_i\cdot a_j)
\otimes b_j \otimes b_i -  a_j \otimes
a_i\otimes (b_i \cdot b_j) \rangle= 0\\
&\Longrightarrow \sum_{i,j=1}^n \Big(
\langle g, f(a_i) h(b_j )  (b_i \cdot a_j ) \rangle - \langle f, g(b_j) h(b_i)(a_i\cdot a_j)\rangle
- \langle h, f(a_j) g(
a_i) (b_i \cdot b_j) \rangle\Big)= 0
 \end{align*} 
Thus, since $r$ is skew-symmetric, by Remark \ref{rmk: r skew-symmetric R} it follows that, for all $f,g,h\in \A^*$,
\begin{align*}
 & -\langle g, \mathcal{R}(f) \mathcal{R}(h)\rangle - \langle f, \mathcal{R}(h)\mathcal{R}(g)\rangle - \langle h, \mathcal{R}(g)\mathcal{R}(f)\rangle= 0,
 \end{align*}
which proves Condition \eqref{eq: R LRYBE}.
The converse is proved similarly, completing the proof. 
\end{proof}
\begin{rmk}
    We usually write the cyclic Equation \ref{eq: R LRYBE} in the following concise form
     \begin{equation*}
        \sum_{cycl}\langle f, \mathcal{R}(g) \mathcal{R}(h) \rangle=0,
    \end{equation*}
for all  $f,g,h \in \A^\ast$.
\end{rmk}

Next, we rewrite the conditions of Theorem \ref{teo: coboundary NALR-algebra} in terms of the map $\mathcal{R}$.


\begin{pro}
Let $(\A, \cdot)$  be a nearly associative  $L$-algebra, $r \in \A \otimes \A$ be skew-symmetric and $\mathcal{R}$ be the linear map defined by  \eqref{eq: R}. If $\Delta_r$ is the comultiplication defined by \eqref{eq:comultiplication r}, then the triple $(\A,\cdot, \Delta_r)$ is a nearly associative bialgebra if and only if  the following conditions are satisfied:
\begin{enumerate}
\item $ \displaystyle \sum_{\sigma \in S_2} \sgn (\sigma) \,\Big(\Bigl\langle f_{\sigma(1)}, \mathcal{R}(\rmL_x^\ast (f_{\sigma(2)})) \mathcal{R}(f_3) \Bigr\rangle
+\Bigl\langle \rmL_x^\ast (f_{\sigma(2)}), \mathcal{R}(f_3) \mathcal{R}(f_{\sigma(1)}) \Bigr\rangle
+\Bigl\langle f_3, \mathcal{R}(f_{\sigma(1)}) \mathcal{R}(\rmL_x^\ast (f_{\sigma(2)})) \Bigr\rangle \Big) =0$,
\vspace{1mm}

\item $\displaystyle   \Bigl\langle f_1, \ad_x(\calR(f_3))\calR(f_2) \Bigr\rangle  - \Bigl\langle f_1, \calR(\ad_x^\ast(f_3)) \calR(f_2) \Bigr\rangle  + \Bigl\langle f_2, \calR(f_1) \ad_x(\calR(f_3)) \Bigr\rangle - \Bigl\langle f_2, \calR(f_1) \calR(\ad_x^\ast (f_3)) \Bigr\rangle=0$,

\vspace{1mm}

\item $\displaystyle \Bigl\langle \rmL_y^\ast (f_2), \mathcal{R}(\rmR_x^\ast (f_1))  \Bigr\rangle
 -\Bigl\langle \rmR_y^\ast (f_2), \mathcal{R}(\rmL_x^\ast (f_1))  \Bigr\rangle= 0$,

 \vspace{1mm}

\item $\displaystyle \Bigl\langle f_1, \calR \big(\ad_y^\ast L_x^\ast (f_2)\big) \Bigr\rangle - \Bigl\langle f_2, \calR \big(\rmL_x^\ast \ad_y^\ast(f_1)\big) \Bigr\rangle  =0$,

\vspace{1mm}

\item $\displaystyle \Bigl\langle \rmR_x^\ast (f_2), \mathcal{R}(\rmR_y^\ast (f_1))  \Bigr\rangle
 -\Bigl\langle \rmL_x^\ast (f_2), \mathcal{R}(\rmL_y^\ast (f_1))  \Bigr\rangle= 0$,

\vspace{1mm}

\item $\displaystyle \sum_{\sigma \in S_2} \sgn (\sigma) \, \Bigl\langle f_{\sigma(1)}, \calR \big(\ad_{x\cdot y}^\ast(f_{\sigma(2)})\big) \Bigr\rangle =0$,
  \end{enumerate}
for all $ x,y \in \A$ and $\displaystyle f_1, f_2 ,f_3 \in \A^\ast $, where $S_2$ is the symmetric group of order $2$ and $\sgn(\sigma)$ is the sign of the permutation $\sigma\in S_2$.
\end{pro}

\begin{proof}
 Consider a skew-symmetric element $r = \sum_{i=1}^n a_i \otimes b_i \in \A \otimes \A$  and $\mathcal{R}$ the linear map defined by  \eqref{eq: R}.
First, we rewrite Condition (1) of Theorem \ref{teo: coboundary NALR-algebra} in terms of  $\mathcal{R}$.  Recall that
\begin{equation*}
\begin{split}
&\mathbf{LR}(r) =  \sum_{i,j=1}^n \Big( a_i \otimes  (b_i \cdot a_j ) \otimes b_j - (a_i\cdot a_j)
\otimes b_j \otimes b_i -  a_j \otimes
a_i\otimes (b_i \cdot b_j) \Big).
\end{split}
\end{equation*}
Hence, for  $ x \in \A$, 
\begin{equation*}
 \begin{split}
 & \big((\rmL_x \otimes \mbox{I} \otimes \mbox{I})- (\tau \otimes \mbox{I})(\rmL_x \otimes \mbox{I} \otimes \mbox{I})\big) \mathbf{LR}(r)=\\
 &\quad \quad \quad   = \sum_{i,j=1}^n
\Big( (x \cdot a_i) \otimes  (b_i \cdot a_j ) \otimes b_j - \big(x \cdot (a_i\cdot a_j)\big)
\otimes b_j \otimes b_i -  (x \cdot a_j) \otimes
a_i\otimes (b_i \cdot b_j) \Big) \\
& \quad \quad \quad \quad 
 + \sum_{i,j=1}^n
\Big( -(b_i \cdot a_j ) \otimes (x \cdot a_i) \otimes b_j + b_j \otimes \big(x \cdot (a_i\cdot a_j)\big)
\otimes  b_i +  a_i\otimes (x \cdot a_j) \otimes
 (b_i \cdot b_j) \Big).
\end{split}
\end{equation*}
Consequently,  since $r$ is skew-symmetric, by Remark \ref{rmk: r skew-symmetric R}, and taking into account  \eqref{eq: def dual maps}, for all $\displaystyle f, g ,h \in \A^\ast $, $x\in \A$, we have 
\begin{equation*}
 \begin{split}
  (f\otimes  g \otimes h) & \Big( \big((\rmL_x \otimes \mbox{I} \otimes \mbox{I})- (\tau \otimes \mbox{I})(\rmL_x \otimes \mbox{I} \otimes \mbox{I})\big) \mathbf{LR}(r) \Big) =\\
 = & \sum_{i,j=1}^n \Big( \Bigl\langle g, f(x \cdot a_i) h(b_j)  (b_i \cdot a_j )    \Bigr\rangle
-\Bigl\langle   f, g( b_j)h( b_i)\big(x \cdot (a_i\cdot a_j)\big)
     \Bigr\rangle
-\Bigl\langle   h,  f(x \cdot a_j)g(
a_i ) (b_i \cdot b_j)\Bigr\rangle \Big)\\
& + \sum_{i,j=1}^n \Big( - \Bigl\langle f,  g(x \cdot a_i) h(b_j ) (b_i \cdot a_j )   \Bigr\rangle
+\Bigl\langle  g,  f( b_j) h( b_i )  \big(x \cdot (a_i\cdot a_j)\big)
   \Bigr\rangle
+\Bigl\langle  h, f(   a_i)g (x \cdot a_j)  
 (b_i \cdot b_j)\Bigr\rangle  \Big) \\
 = & -\Bigl\langle g, \mathcal{R}(\rmR_x^\ast (f)) \mathcal{R}(h) \Bigr\rangle
-\Bigl\langle \rmR_x^\ast (f), \mathcal{R}(h) \mathcal{R}(g) \Bigr\rangle
-\Bigl\langle h, \mathcal{R}(g) \mathcal{R}(\rmR_x^\ast (f)) \Bigr\rangle\\
& + \Bigl\langle f, \mathcal{R}(\rmR_x^\ast (g)) \mathcal{R}(h) \Bigr\rangle
+\Bigl\langle \rmR_x^\ast (g), \mathcal{R}(h) \mathcal{R}(f) \Bigr\rangle
+\Bigl\langle h, \mathcal{R}(f) \mathcal{R}(\rmR_x^\ast (g)) \Bigr\rangle .
\end{split}
\end{equation*}
Then (1) of Theorem \ref{teo: coboundary NALR-algebra} is equivalent to the  condition
\begin{equation*}
 \begin{split}
 &   -\Bigl\langle g, \mathcal{R}(\rmL_x^\ast (f)) \mathcal{R}(h) \Bigr\rangle
-\Bigl\langle \rmL_x^\ast (f), \mathcal{R}(h) \mathcal{R}(g) \Bigr\rangle
-\Bigl\langle h, \mathcal{R}(g) \mathcal{R}(\rmL_x^\ast (f)) \Bigr\rangle\\
&  + \Bigl\langle f, \mathcal{R}(\rmL_x^\ast (g)) \mathcal{R}(h) \Bigr\rangle
+\Bigl\langle \rmL_x^\ast (g), \mathcal{R}(h) \mathcal{R}(f) \Bigr\rangle
+\Bigl\langle h, \mathcal{R}(f) \mathcal{R}(\rmL_x^\ast (g)) \Bigr\rangle =0.
\end{split}
\end{equation*}

Now we deal with Condition (2) of Theorem  \ref{teo: coboundary NALR-algebra}. For $x\in \A$, we have 
\begin{align*}
\Big( (m & (\ad_x \otimes \mbox{I}) \otimes \mbox{I} \otimes \mbox{I} + m \otimes \mbox{I} \otimes \ad_x)(\mbox{I} \otimes \xi)- (\mbox{I} \otimes m (\mbox{I} \otimes \ad_x) \otimes \mbox{I} + \mbox{I} \otimes m \otimes \ad_x)\Big)(r \otimes r)=\\
        =&\sum_{i,j=1}^n \Big(
       ((x\cdot a_i)\cdot a_j)\otimes b_j \otimes b_i 
         -  (( a_i \cdot x)\cdot a_j) \otimes b_j  \otimes b_i+(a_i\cdot a_j) \otimes b_j \otimes (x\cdot b_i) 
         -(a_i \cdot a_j) \otimes b_j \otimes ( b_i \cdot x) \Big)\\
         &  +\sum_{i,j=1}^n \Big( 
          a_j \otimes ( b_j \cdot ( a_i \cdot x))\otimes b_i 
          - a_j \otimes (b_j \cdot (x\cdot a_i))\otimes b_i+
          a_j \otimes (b_j \cdot a_i) \otimes ( b_i \cdot x)
        -  a_j \otimes ( b_j \cdot a_i) \otimes (x\cdot b_i)\Big).
\end{align*}
Consequently, for all $\displaystyle f, g ,h \in \A^\ast $, $x\in \A$, we have 
\begin{align*}
(f \otimes & g \otimes h) \Big(\Big( (m (\ad_x \otimes \mbox{I}) \otimes \mbox{I} \otimes \mbox{I} + m \otimes \mbox{I} \otimes \ad_x)(\mbox{I} \otimes \xi)- (\mbox{I} \otimes m (\mbox{I} \otimes \ad_x) \otimes \mbox{I} + \mbox{I} \otimes m \otimes \ad_x)\Big)(r \otimes r)\Big)=\\
     =&\sum_{i,j=1}^n \Big(
     \Bigl\langle f, g(b_j)h(b_i)((x\cdot a_i)\cdot a_j) \Bigr\rangle
         -  \Bigl\langle f, g(b_j)h(b_i)  (( a_i \cdot x)\cdot a_j) \Bigr\rangle
         + \Bigl\langle f, g(b_j)h(x\cdot b_i)(a_i\cdot a_j)  \Bigr\rangle
         - \Bigl\langle f,  g(b_j)h(b_i\cdot x)(a_i\cdot a_j) \Bigr\rangle \Big) \\
         &  +\sum_{i,j=1}^n \Big( 
          \Bigl\langle g, f(a_j) h(b_i) ( b_j \cdot ( a_i \cdot x))\Bigr\rangle
          -  \Bigl\langle g, f(a_j) h(b_i)  (b_j \cdot (x\cdot a_i))\Bigr\rangle
          +  \Bigl\langle g, f(a_j) h(b_i \cdot x) (b_j \cdot a_i)\Bigr\rangle
        -  \Bigl\langle g, f(a_j) h(x \cdot b_i)( b_j \cdot a_i)\Bigr\rangle \Big)\\
        =& \Bigl\langle f, \ad_x(\calR(h))\calR(g) \Bigr\rangle  - \Bigl\langle f, \calR(\ad_x^\ast(h)) \calR(g) \Bigr\rangle  + \Bigl\langle g, \calR(f) \ad_x(\calR(h)) \Bigr\rangle - \Bigl\langle g, \calR(f) \calR(\ad_x^\ast (h)) \Bigr\rangle.
\end{align*}
Thus Condition (2) of Theorem \ref{teo: coboundary NALR-algebra}  is equivalent to
\begin{equation*}
    \Bigl\langle f, \ad_x(\calR(h))\calR(g) \Bigr\rangle  - \Bigl\langle f, \calR(\ad_x^\ast(h)) \calR(g) \Bigr\rangle  + \Bigl\langle g, \calR(f) \ad_x(\calR(h)) \Bigr\rangle - \Bigl\langle g, \calR(f) \calR(\ad_x^\ast (h)) \Bigr\rangle=0.
\end{equation*}

Now, to rewrite Condition (3) of Theorem \ref{teo: coboundary NALR-algebra} in terms of  $\mathcal{R}$, we recall that,
 for  $ x,y \in \A$,
\begin{equation*}
 \begin{split}
 & \big((\rmL_x \otimes \rmR_y)- (\rmR_x \otimes \rmL_y)\big)  (r)=  \sum_{i =1}^n
\big( (x \cdot a_i) \otimes  (b_i \cdot y )   -  (a_i \cdot x) \otimes
(y \cdot b_i)  \big).
\end{split}
\end{equation*}
Hence, for all $\displaystyle f, g   \in \A^\ast $,
\begin{equation*}
 \begin{split}
 & (f\otimes g ) \Big( \big((\rmL_x \otimes \rmR_y)- (\rmR_x \otimes \rmL_y)\big)  (r)  \Big)= 
\Bigl\langle \rmL_y^\ast (g), \mathcal{R}(\rmR_x^\ast (f))  \Bigr\rangle
 -\Bigl\langle \rmR_y^\ast (g), \mathcal{R}(\rmL_x^\ast (f))  \Bigr\rangle.
\end{split}
\end{equation*}
 Then Condition (3) of Theorem \ref{teo: coboundary NALR-algebra} is equivalent to
 \begin{eqnarray*}
\displaystyle \Bigl\langle \rmL_y^\ast (g), \mathcal{R}(\rmR_x^\ast (f))  \Bigr\rangle
 -\Bigl\langle \rmR_y^\ast (g), \mathcal{R}(\rmL_x^\ast (f))  \Bigr\rangle= 0,&&
\end{eqnarray*}
Similarly, Condition (5) of Theorem \ref{teo: coboundary NALR-algebra} is equivalent to
 \begin{eqnarray*}
\displaystyle \Bigl\langle \rmR_x^\ast (g), \mathcal{R}(\rmR_y^\ast (f))  \Bigr\rangle
 -\Bigl\langle \rmL_x^\ast (g), \mathcal{R}(\rmL_y^\ast (f))  \Bigr\rangle= 0.&&
\end{eqnarray*}

Now to express Condition (4) of Theorem \ref{teo: coboundary NALR-algebra} in terms of  $\mathcal{R}$, we have 
 for  $ x,y \in \A$, 
\begin{equation*}
 \begin{split}
 & (\ad_y \rmR_x  \otimes \mbox{I} - \mbox{I} \otimes  \rmR_x \ad_y  )(r)=  
 \sum_{i=1}^n \Big( \big(y\cdot (a_i \cdot x)\big) \otimes b_i
  - \big(( a_i \cdot x)\cdot y \big) \otimes b_i
 -a_i \otimes \big((y\cdot b_i)\cdot x \big)  +a_i \otimes \big(( b_i \cdot y)\cdot x \big)\Big).
\end{split}
\end{equation*}
Then, for all $\displaystyle f, g   \in \A^\ast $,
\begin{equation*}
 \begin{split}
 & (f\otimes g  ) \Big( (\ad_y \rmR_x  \otimes \mbox{I} - \mbox{I} \otimes  \rmR_x \ad_y  )(r) \Big) =  - \Bigl\langle g, \calR \big(\rmL_x^\ast \ad_y^\ast(f)\big) \Bigr\rangle  +  \Bigl\langle f, \calR \big(\ad_y^\ast L_x^\ast (g)\big) \Bigr\rangle.
\end{split}
\end{equation*} 
Thus (4) of Theorem \ref{teo: coboundary NALR-algebra} is equivalent to the condition
\begin{equation*}
   \Bigl\langle f, \calR \big(\ad_y^\ast L_x^\ast (g)\big) \Bigr\rangle - \Bigl\langle g, \calR \big(\rmL_x^\ast \ad_y^\ast(f)\big) \Bigr\rangle  =0.
\end{equation*}

Finally, we deal with Condition (6) of Theorem \ref{teo: coboundary NALR-algebra}. Since for $x,y\in \A$ 
\begin{equation*}
    (\mbox{I} \otimes \mbox{I} -\tau)(  \ad_{x\cdot y}  \otimes \mbox{I} )(r)=  \sum_{i=1}^n \Big( \big((x\cdot y)\cdot a_i \big) \otimes b_i
  - \big(a_i \cdot(x\cdot y)\big) \otimes b_i
 + a_i \otimes \big((x\cdot y)\cdot b_i \big)
 - a_i \otimes \big(b_i \cdot(x\cdot y)\big) 
 \Big),
\end{equation*}
we have for all $\displaystyle f, g   \in \A^\ast $,
\begin{align*}
    (f\otimes g) \Big((\mbox{I} \otimes \mbox{I} -\tau)(  \ad_{x\cdot y}  \otimes \mbox{I} )(r)\Big)
    = -  \Bigl\langle g, \calR\big(\ad_{x\cdot y}^\ast (f)\big) \Bigr\rangle + \Bigl\langle f, \calR\big(\ad_{x\cdot y}^\ast(g)\big) \Bigr\rangle.
\end{align*}
Then Condition (6) of Theorem \ref{teo: coboundary NALR-algebra} is equivalent to
\begin{equation*}
     \Bigl\langle f, \calR \big(\ad_{x\cdot y}^\ast(g)\big) \Bigr\rangle  -  \Bigl\langle g, \calR \big(\ad_{x\cdot y}^\ast (f)\big) \Bigr\rangle =0,
\end{equation*}
completing the proof.
\end{proof}


 \end{document}